\documentclass[preprint,12pt]{elsarticle}

\usepackage{graphicx, amssymb, amsmath, epsfig}
\usepackage{color}
\usepackage{amsthm}
\usepackage{soul} 
\usepackage{amsfonts}%
\usepackage{booktabs}
\usepackage{boxedminipage}
\usepackage{amsfonts}
\usepackage{algorithmic}
\usepackage{algorithm,booktabs}
\usepackage[mathscr]{eucal}
\usepackage{todonotes}

\usepackage{mathtools}
\DeclarePairedDelimiter{\ceil}{\lceil}{\rceil}

\newtheorem{teo}{Theorem}[section]
\newtheorem{rmk}{Remark}[section]

\newcommand{\ep}{\varepsilon_{\mathcal{T}}}
\newcommand{\R}{\mathbb{R}}

\DeclareMathOperator*{\argmin}{arg\,min}

\newcommand{\second}[2][black]{\emph{\textcolor{#1}{#2}}}

\journal{Applied Numerical Mathematics}

\begin{document}

\begin{frontmatter}



\title{A HJB-POD approach for the control of \\nonlinear PDEs on a tree structure}


\author[label1]{Alessandro Alla}
\author[label2]{Luca Saluzzi}

\address[label1]{Department of mathematics, PUC-Rio, Rio de Janeiro, Brazil, alla@mat.puc-rio.br}
\address[label2]{Department of mathematics, Gran Sasso Science Institute, L'Aquila, Italy, luca.saluzzi@gssi.it}

\begin{abstract}
The Dynamic Programming approach allows to compute a feedback control for nonlinear problems, but suffers from the {\em curse of dimensionality}. The computation of the control relies on the resolution of a nonlinear PDE, the Hamilton-Jacobi-Bellman equation, with the same dimension of the original problem. Recently, a new numerical method to compute the value function on a tree structure has been introduced. The method allows to work without a structured grid and avoids any interpolation.

Here, we aim at testing the algorithm for nonlinear two dimensional PDEs. We apply model order reduction to decrease the computational complexity since the tree structure algorithm requires to solve many PDEs. Furthermore, we prove an error estimate which guarantees the convergence of the proposed method. Finally, we show efficiency of the method through numerical tests.

\end{abstract}

\begin{keyword}
Optimal control, Hamilton-Jacobi-Bellman equation, Model order reduction, Proper Orthogonal Decomposition, Tree structure, Error Estimates
\MSC[2010]{49L20, 49L25, 49J20,78M34,65N99,62H25}



\end{keyword}

\end{frontmatter}


\section{Introduction}
\label{sec:intro}


The dynamic programming (DP) approach, introduced by Bellman in the late '50, allows to obtain a feedback control by means of the knowledge of the value function. Thus, we solve a nonlinear Partial Differential Equation (PDE) known as Hamilton-Jacobi-Bellman (HJB) equation that has the same dimension of the optimal control problem. It is well-known that this problem suffers from the {\em curse of dimensionality}: typically this equation has to be solved on a space grid and this is the major bottleneck for numerical methods in high-dimension. We refer to \cite{BCD97} and \cite{FF13} for a complete description of theoretical and numerical results, respectively. 

The focus of this paper is to solve finite horizon optimal control problems for nonlinear PDEs. It is straightforward to understand the difficulty of the problem when dealing with a DP approach, since the discretization of PDEs leads to a very large system of ODEs, which makes the problem not feasible on a structured grid. In the literature several methods have been introduced to mitigate the curse of dimensionality. Although a complete description of numerical methods for HJB goes beyond the scopes of this work, we distinguish between numerical methods for the control of ODEs and PDEs via the HJB equation. In the former we mention, among others, domain decomposition methods and iterative schemes based on a semi-Lagrangian approach (see e.g. \cite{CCFP12,AFK15} and the references therein). On the other hand, to compute feedback control of PDEs, it is very common the use of model order reduction techniques to reduce the complexity of the system and, therefore, the dimension of the corresponding HJB equation. In particular, we refer to the Proper Orthogonal Decomposition (POD, see e.g. \cite{Vol11}) which will constitute one of the building blocks for the current paper. The POD method allows to compute low-rank orthogonal projectors by means of Singular Value Decomposition (SVD) upon snapshots of the dynamical system at given time instances. This is a serious issue of this approach for optimal control problems since the control input is not known in advance and it is usually necessary to plug a forecast to compute the snapshots. However, on a structure grid, POD has been successfully coupled with the HJB approach for the control of PDEs. We refer to the pioneering work \cite{KVX04} and to \cite{AFV17} for error estimates of the method. We note that this approach is only a mitigation of the curse of dimensionality because it is not possible to work with a reduced space with dimension larger than $5$ and the aim of the POD method is to make the problem feasible even for very high dimensional equation such as PDEs. Other approaches to mitigate the curse of dimensionality are built upon the sparse grid method (see e.g. \cite{GK17}), the spectral elements method (see \cite{KK18}) and, more recently, a tensor decomposition (see \cite{DKK19}). For the sake of completeness, we mention that the control of PDEs can be solved with other methods such as, among others, open loop techniques (see e.g \cite{HPUU09}) and model predictive control (see e.g \cite{GP11}).

Recently, in \cite{AFS18} the authors proposed a new method based on a time discretization of the dynamics which allows to mimic the discrete dynamics in high-dimension via a tree structure, considering a discretized control space. The method deals with a finite horizon optimal control problem and, in the discretization, the tree structure replaces the space grid which allows to increase the dimension of the state space. However, the tree structure complexity increases exponentially due to the number of time steps and control inputs. To decrease the complexity of the tree a pruning technique has been implemented to reduce the number of branches in the tree obtaining rather accurate results. Error estimates for the method can be found in the recent work \cite{SAF18}. Therefore, it is clear that the method is expensive when we deal with PDEs since it requires to solve many equations for several control inputs. 
It is then natural to couple the TSA with POD in order to speed up the method. With the approach studied in the current paper we have four major advantages:
\begin{enumerate} 
\item we build the snapshots set upon all the trajectories that appear in the tree, avoiding the selection of a forecast for the control inputs which is always not trivial for model reduction,
\item the application of POD also allows an efficient pruning since it reduces the dimension of the problem and provides information on the most variable components,
\item the theory of DPP is valid on the whole state space $\R^d$ but, in general, for numerical reasons we need to restrict our equation to a bounded domain. Our method avoids to define the numerical domain for the projected problem, which is a difficult task since we lose the physical meaning of the reduced coordinates,
\item we are not restricted to consider a reduced space dimension smaller than $5$ as in e.g. \cite{KVX04, AFV17},  which was a limitation of the method since many classes of PDEs require more basis functions to capture the essential features.
\end{enumerate}

Finally, we remark that to obtain a low-dimensional problem completely independent from the dimension of the original system, we use the Discrete Empirical Interpolation Method as in \cite{CS10}. To validate our approach we also provide a-priori error estimate for the coupling between TSA and model order reduction.

The paper is organized as follows: we define the optimal control problem and the DP approach in Section \ref{sec:ocp}. We recall the tree structure algorithm in Section \ref{sec:tsa} and the POD method in Section \ref{Sec:pod}. In Section \ref{sec:coupl} we present, step by step, the coupling between POD and the TSA, and in Section \ref{sec:err} we provide an error estimate for the coupled method. Finally, numerical tests for two-dimensional nonlinear PDEs are shown in Section \ref{sec:tests}. We give our conclusions and perspectives in Section \ref{sec:con}.

\section{The optimal control problem}\label{sec:ocp}

In this section we describe the optimal control problem and the essential features of the DP approach.
 Let us consider a large system of ordinary differential equations in the following form:
\begin{equation}
\left\{ \begin{array}{ll}\label{ode}
\dot{{ y}}(s)&={ A}{ y}(s)+{ F}(s, { y}(s))+{ B} u(s),\;\; s\in(t,T],\\
 y(t)& =x,\\
\end{array} \right.
\end{equation}
where $x\in\R^d$ is a given initial data, $ A \in \R^{d\times d}, { B}\in \R^{d\times m}$ are given matrices and ${ F}:[t,T]\times\R^d\rightarrow\R^d$ is a continuous function in both arguments and locally Lipschitz-type with respect to the second variable. We will denote by $y:[t,T]\rightarrow\R^d$ the solution, by $u:[t,T]\rightarrow\R^m$ the control and by
\[\mathcal{U}=\{u:[t,T]\rightarrow U, \mbox{measurable} \}
\]
the set of admissible controls where $U\subset \R^m$ is a compact set. We will assume that there exists a unique solution for \eqref{ode} for each $u\in\mathcal{U}$. Whenever we want to stress the dependence on the control $u$, we write $y(s,u)$.

 This wide class of problems arises in many applications, especially from the numerical approximation of PDEs. In such cases, the dimension of the problem is the number of spatial grid points used for the discretization and it can be very large. 

To ease the notation we will denote the right hand side as follows: 
\begin{equation}
f(y(s),u(s),s):={ A}{ y}(s)+{ F}(s,{ y}(s))+{ B} u(s).
\end{equation}
To select the optimal trajectory, we consider the following cost functional
\begin{equation}\label{cost}
 J_{x,t}(u):=\int_t^T L(y(s,u),u(s),s)e^{-\lambda (s-t)}\, ds+g(y(T,u))e^{-\lambda (T-t)},
\end{equation}
where $L:\R^d\times\R^m\times [t,T]\rightarrow\R$ is the running cost, $g:\R^d \rightarrow \R$ is the final cost and $\lambda\geq0$ is the discount factor. We will suppose that the functions $L(\cdot, u,t )$ and $g(\cdot)$ are Lipschitz continuous. The optimal control problem then reads:
\begin{equation}\label{ocp:full}
\inf_{u\in \mathcal{U}} J_{x,t}(u) \mbox{ s.t. } y(s) \mbox{ satisfies } \eqref{ode}.
\end{equation}

The final goal is the computation of the control in feedback form $u(s)=\eta(y(s),s),$ in terms of the state equation $y(s),$ where $\eta$ is the feedback map. To derive optimality conditions, we use the Dynamic Programming Principle (DPP). We first define the value function 
\begin{equation}
v(x,t):=\inf\limits_{u\in\mathcal{U}} J_{x,t}(u).
\label{value_fun}
\end{equation}
Given the above assumptions, the value function $v$ is bounded and continuous in $\R^d\times [t,T]$ and it satisfies the DPP, i.e. for every $\tau\in [t,T]$: 
\begin{equation}\label{dpp}
v(x,t)=\inf_{u\in\mathcal{U}}\left\{\int_t^\tau L(y(s),u(s),s) e^{-\lambda (s-t)}ds+ v(y(\tau),\tau) e^{-\lambda (\tau-t)}\right\}.
\end{equation}
Due to \eqref{dpp}, we can derive the HJB equation for every $(x,s)\in\R^d\times [t,T)$:
\begin{equation}\label{HJB}
\left\{
\begin{array}{ll} 
&\dfrac{\partial v}{\partial s}(x,s) -\lambda v(x,s)+ \min\limits_{u\in U }\left\{L(x, u, s)+ \nabla v(x, s) \cdot f(x,u, s)\right\} = 0, \\
&v(x,T) = g(x).
\end{array}
\right.
\end{equation}
We refer to \cite{BCD97} for more details on the topic. Once the value function has been computed, it is possible to obtain the optimal feedback control as:
\begin{equation}\label{feedback}
u^*(s):=  \argmin_{u\in U }\left\{L(x,u,s)+ \nabla v(x,s) \cdot f(x,u,s)\right\}. 
\end{equation}

\subsection{Dynamic Programming on a Tree Structure}
\label{sec:tsa}	

In this section we will recall the {\em finite horizon control problem} and its approximation by the tree structure algorithm (see \cite{AFS18} for a complete description of the method and \cite{SAF18} for theoretical results). 
The computation of analytical solutions of Equation \eqref{HJB} is a difficult task due to its nonlinearity and approximation techniques should take in consideration discontinuities in the gradient (see \cite{FF13} and the references therein). Here, we discretize equation \eqref{HJB}, only partitioning the time interval $[t,T]$ with step size $\Delta t: = (T-t)/\overline N$, where $\overline{N}$ is the total number of steps. Thus, for $n= \overline{N}-1,\dots, 0$ and every $x\in \R^d$, we have 
\begin{equation}\label{SL}
\left\{
\begin{array}{ll}
V^{n}(x)&=\min\limits_{u\in U}[\Delta t\,L(x, u, t_n)+e^{-\lambda \Delta t}V^{n+1}(x+\Delta t f(x, u, t_n))], \\
V^{\overline{N}}(x)&=g(x).
\end{array}
\right.
\end{equation}
where $t_n=t+ n \Delta t,\, t_{\overline N} = T$ and $V^n(x):=V(x, t_n).$

The term $V^{n+1}(x+\Delta t f(x, u, t_n))$ is usually computed by interpolation on a grid, since $x+\Delta t f(x, u, t_n)$ is in general not a grid point. To avoid the use of interpolation, we build a non-structured grid with a tree structure. 

We first discretize the control domain into $M$ discrete controls and to ease the notation, in what follows, we keep denoting $U$ also the discrete set of controls. The tree will be denoted by $\mathcal{T}:=\cup_{j=0}^{\overline{N}} \mathcal{T}^j,$ where each level $\mathcal{T}^j$ contains all the nodes of the tree at time $t_j$. We proceed as follows: first we start from the initial state $x$, which will form the first level $\mathcal{T}^0$. Then, we follow the discrete dynamics, given e.g. by an explicit Euler scheme, inserting the discrete control $u_j \in U$, obtaining 
$$\zeta_j^1 = x+ \Delta t \, f(x,u_j,t),\qquad j=1,\ldots,M.$$ 
Therefore, we have $\mathcal{T}^1 =\{\zeta_1^1,\ldots, \zeta^1_M\}$. We can characterize the nodes by their $n-$th {\em time level} as follows
$$\mathcal{T}^n = \{ \zeta^{n-1}_i + \Delta t f(\zeta^{n-1}_i, u_j,t_{n-1}) \}_{j=1}^{M},\quad i = 1,\ldots, M^{n-1},$$
and  the tree can be shortly defined as
 $$\mathcal{T}:= \{ \zeta_j^n  \}_{j=1}^{M^n},\quad n=0,\ldots \overline{N},$$ 
where the nodes $\zeta^n_i$ are obtained following the dynamics at time $t_n$ with the controls $\{u_{j_k}\}_{k=0}^{n-1}$:
\begin{equation*}
\begin{array}{ll}
\zeta_{i_n}^n &= \zeta_{i_{n-1}}^{n-1} + \Delta t f(\zeta_{i_{n-1}}^{n-1}, u_{j_{n-1}},t_{n-1})\\
&= x+ \Delta t \sum_{k=0}^{n-1} f(\zeta^k_{i_k}, u_{j_k},t_k), 
\end{array}
\end{equation*}
with $\zeta^0 = x$, $i_k = \ceil[\bigg]{\dfrac{i_{k+1}}{M}}$ and $j_k\equiv i_{k+1} \mbox{mod } M$, where $\ceil{\cdot}$ is the ceiling function.

The cardinality of tree increases exponentially, i.e. $|\mathcal{T}| =O(M^{\overline{N}})$,
where $M$ is the number of controls and $\overline{N}$ the number of time steps. 
To miti\-gate this problem, we consider the following pruning rule: given a threshold $\ep>0$, we can cut off a new node $\zeta^n_i$, if it verifies the following condition with a certain $\zeta^n_j$
\begin{equation}\label{tol_cri}
\begin{array}{cc}
\Vert \zeta^n_i-\zeta^n_j \Vert \le \ep,
\mbox{ for  }i\ne j \mbox{ and } n = 0,\ldots, \overline{N}.
\end{array}
\end{equation}

The pruning rule \eqref{tol_cri} helps to save a huge amount of memory. If we choose the tolerance properly, e.g. $\ep = O(\Delta t^2)$, we keep the same accuracy of the approach without pruning, as shown in \cite{SAF18}. To increase the order of convergence, one could use a higher order method for the discretization of the ODE \eqref{ode}. More details can be found in \cite{AFS18b}.


The computation of the numerical value function $V(x,t)$ will be done on the tree nodes
\begin{equation}\label{num:vf}
V(x,t_n)=V^n(x), \quad \forall x \in \mathcal{T}^n, 
\end{equation}
and it follows directly from the DPP. The tree $\mathcal{T}$ will form the spatial grid and we can write a time discretization for \eqref{HJB} as follows:
\begin{equation}
\begin{cases}
V^{n}(\zeta^n_i)= \min\limits_{u\in U} \{e^{-\lambda \Delta t} V^{n+1}(\zeta^n_i+\Delta t f(\zeta^n_i,u,t_n)) +\Delta t \, L(\zeta^n_i,u,t_n) \}, \\
 \qquad \qquad\qquad \qquad\qquad  \zeta^n_i \in \mathcal{T}^n\,, n = \overline{N}-1,\ldots, 0, \\
V^{\overline{N}}(\zeta^{\overline{N}}_i)= g(\zeta_i^{\overline{N}}), \qquad\qquad \qquad\qquad \qquad\qquad   \zeta_i^{\overline{N}} \in \mathcal{T}^{\overline{N}}.
\end{cases}
\label{HJBt2}
\end{equation}

Since the control set $U$ is discrete, the minimization is computed by comparison. We refer to \cite{B73,KKK16} for other techniques to compute the minimization in \eqref{HJBt2}. A detailed comparison and discussion about the classical method and tree structure algorithm can be found in \cite{AFS18}, whereas the interested reader will find in \cite{SAF18} the error estimates for the proposed algorithm.
The computation of the feedback on a tree structure takes advantage of the discrete control set and therefore during the computation of the value function, we can store the indices which provide the optimal trajectory. More details on the computation of the feedback control are given in Section \ref{sec:coupl}.

\section{Model order reduction and POD method}  \label{Sec:pod}

In this section we first recall the POD method for the state equation \eqref{ode} and later how to apply it to reduce the dimension of the optimal control problem \eqref{ocp:full}.
\subsection{POD for the state equation}

The solution of the system \eqref{ode} may be very expensive and it is useful to deal with projection techniques to reduce the complexity of the problem. Although a complete description of model order reduction methods goes beyond the scopes of this work, here we recall the POD method. We refer the interested reader to e.g. \cite{Sir87, Vol11} for more details on the topic and to \cite{BGW15} for a review of different projection techniques.

Let us assume we have computed a numerical (or analytical if possible) solution of \eqref{ode} on the time grid points $t_j$, $j\in \{0,\ldots,N\}$ for some given control inputs. Then, we collect the {\em snapshots} $\{y(t_i)\}_{i=0}^N$ into the matrix ${ Y} = [{ y}(t_0),\ldots, { y}(t_N)] \in \R^{d \times (N+1) }$. The aim of the method is to determine a POD basis $ \Psi=\{\psi_1,\ldots,\psi_\ell\}$ of rank $\ell\ll \min\{d,N+1\}$ to describe the set of data collected in time by solving the following minimization problem:
\begin{equation}\label{pbmin}
\min_{ {{\psi}}_1,\ldots,{{\psi}}_\ell\in\R^d} \sum_{j=0}^N \left|{ y}(t_j)-\sum_{i=1}^\ell \langle { y}(t_j),{{\psi}}_i\rangle{{\psi}}_i\right|^2\quad \mbox{such that }\langle {{\psi}}_i,{{\psi}}_j\rangle=\delta_{ij}.
\end{equation}
The associated norm is given by the Euclidean inner product $|\cdot|^2=\langle\cdot,\cdot\rangle$. The solution of \eqref{pbmin} is obtained by the SVD of the snapshots matrix ${ Y}= \Psi \Sigma V^T$, where we consider the first $\ell-$columns $\{{{\psi}}_i\}_{i=1}^\ell$ of the orthogonal matrix $\Psi$. The selection of the rank of POD basis is based on the error computed in \eqref{pbmin} which is related to the singular values neglected. We will choose $\ell$ such that $\mathcal{E}(\ell)\approx 0.999$, with
\begin{equation}\label{POD_cri}
\mathcal{E}(\ell)=\dfrac{\sum_{i=1}^\ell \sigma_i^2}{\sum_{i=1}^{\min\{d,N+1\}} \sigma_i^2},
\end{equation}
where $\{\sigma_i\}_{i=1}^{\min\{d,N+1\}}$ are the singular values of $Y$.   

However, the error strongly depends on the quality of the computed snapshots. This is clearly a limit when dealing with optimal control problems, since the control input is not known a-priori and it is necessary to have a reasonable forecast. In Section \ref{sec:coupl} we will explain how to select the control input $u(t)$ to solve \eqref{ocp:full}.

To ease the notation, in what follows, we will denote by $\Psi\in\R^{d\times\ell}$ the POD basis of rank $\ell$. Let us assume that the POD basis $\Psi$ have been computed and make use of the following assumption to obtain a reduced dynamical system: 
\begin{equation}\label{pod_ans}
{ y}(s)\approx { \Psi} { y^\ell}(s),
\end{equation}
where ${ y}^\ell(s)$ is a function from $[t,T]$ to $\R^\ell$. If we plug \eqref{pod_ans} into the full model \eqref{ode} and exploit the orthogonality of the POD basis, the reduced model reads: 

\begin{equation}\label{pod_sys}
\left\{\begin{array}{l}
\dot{{ y}}^\ell(s)={ A}^\ell { y}^\ell(s)+ { \Psi}^T{ F}(s, { \Psi} { y}^\ell(s)) + { B}^\ell u(s),\\
{ y}^\ell(t)={ x^\ell},
\end{array}\right.
\end{equation}
where $ { A}^\ell= { \Psi}^T { A} { \Psi}, { B}^\ell= { \Psi}^T { B}$ and ${ x^\ell}={ { \Psi}}^Tx\in\R^\ell$. We also note that $ { A}^\ell\in\R^{\ell\times\ell}$ and ${ B}^\ell \in\R^{\ell\times m}$. Error estimates for the reduced system \eqref{pod_sys} can be found in \cite{KV02}. In what follows we are going to define the reduced dynamics as:
\begin{equation}
f^\ell(y^\ell(s),u(s),s):= { A^\ell}{ y^\ell}(s)+\Psi^T{ F}(s,\Psi y^\ell(s))+ B^\ell u(s).
\end{equation}
\paragraph{\bf Discrete Empirical Interpolation Method}
The solution of \eqref{pod_sys} is still computationally expensive, since the nonlinear term ${ F}(s,{ \Psi} { y}^\ell(s))$ depends on the dimension of the original problem, i.e. the variable ${ \Psi} { y}^\ell(s)\in\R^d$. To avoid this issue the {\em Empirical Interpolation Method} (EIM, \cite{BMNP04}) and {\em Discrete Empirical Interpolation Method} (DEIM, \cite{CS10}) were introduced. 

The computation of the POD basis functions for the nonlinear part is related to the set of the snapshots $F(t_j,{ y}(t_j))$, where ${ y}(t_j)$ are already computed from \eqref{ode}. We denote by $\Phi\in\R^{d\times k}$ the POD basis functions of rank $k\ll \min\{d,N+1\}$ of the nonlinear part. The DEIM approximation of ${  F}(t,{ y}(t))$ is given in the following form:
\begin{equation}\label{f_DEIM}
{ F}^{\mbox{\tiny DEIM}}(s,{ y}^{\mbox{\tiny DEIM}}(s)):=\Phi({ S}^T \Phi)^{-1} { F}(s,{ y}^{\mbox{\tiny DEIM}}(s)),
\end{equation}
where ${ S}\in\R^{d\times k}$ and \second{${ y}^{\mbox{\tiny DEIM}}(s):={ S}^T{ \Psi}{ y}^\ell(s)$}. Here, we assume that each component of the nonlinearity is independent from each other, \second{i.e. we assume that $F(s,y):=[\bar{F}(s,y_1(s)),\ldots, \bar{F}(s,y_d(s))]$, with $\bar{F}: [t,T] \times\R\rightarrow \R$,} then the matrix $S$ can be moved into the  nonlinearity. Again, we refer to \cite{CS10} for a complete description of the method and extensions to more general nonlinear functions.
  The role of the matrix $S$ is to select interpolation points to evaluate the nonlinearity. The selection is made according to the LU decomposition algorithm with pivoting~\cite{CS10}, or following the QR decomposition with pivoting~\cite{DG15}. We finally note that all the quantities in \eqref{f_DEIM} are independent of the full dimension $d,$ since the quantity $\Psi^T\Phi (S^T\Phi)^{-1}\in \R^{\ell\times k}$ can be precomputed. Typically the dimension $k$ is much smaller than the full dimension. This allows the reduced order model to be completely independent of the full dimension as follows:
\begin{equation}\label{pod_sysdeim}
\left\{\begin{array}{l}
\dot{{ y}}^\ell(s)={ A}^\ell { y}^\ell(s)+{ \Psi}^T { F}^{\mbox{\tiny{ DEIM}}}(s,{ y}^{\mbox{\tiny DEIM}}(s))+{ B}^\ell u(s),\\
{ y}^\ell(t)=x^\ell.
\end{array}\right.
\end{equation}

In what follows, we are going to define the reduced POD-DEIM dynamics as:
\begin{equation}
f^{\ell, \mbox{\tiny DEIM}}(y^\ell(s),u(s),s):={ A^\ell}{ y^\ell}(s)+\Psi^T F^{\mbox{\tiny DEIM}}(s,S^T \Psi y^\ell(s))+ B^\ell u(s).
\end{equation}
The DEIM error is given by:

\begin{equation}\label{DEIM_err}
\|\widetilde{F} - \widetilde F^{\mbox{\tiny DEIM}}\|_2\leq  c\|( I-\Phi\Phi^T) \widetilde{F}\|_2,\quad \,\,\, \mbox{with} \,\,\, c=\|( S^T \Phi)^{-1}\|_2,
\end{equation}
for a given snapshots set $\widetilde F=\{F(t_j,y(t_j))\}_{j=0}^N$ and its DEIM approximation $\widetilde F^{\mbox{\tiny DEIM}}=\Phi({ S}^T \Phi)^{-1} S^T \widetilde F$ as shown in \cite{CS10, DG15}. A further reduction might also be performed by using the dynamic mode decomposition as in \cite{AK17}.

\subsection{POD for the optimal control problem}

The key ingredient to compute feedback control is the knowledge of the value function expressed in \eqref{SL}, which is a nonlinear PDE whose dimension is given by the dimension of \eqref{ode}. It is clear that its approximation is very expensive. Therefore, we are going to apply the POD method to reduce the dimension of the dynamics and then solve the corresponding (reduced) discrete DPP which is now feasible and defined below. Let us first define the reduced running cost and the reduced final cost as
 $$L^\ell(x^\ell,u,s) = L(\Psi x^\ell,u,s), \quad g^\ell(x^\ell) = g(\Psi x^\ell). $$
Next, we introduce the reduced optimal control problem for \eqref{ocp:full}. For a given control $u$, we denote by $y^\ell(s,u)$ the unique solution to \eqref{pod_sysdeim} at time $s$. Then, the reduced cost is given by
\begin{equation}
J_{x^\ell, t}^\ell(u)=\int_t^T L^\ell\big(y^\ell(s,u),u(s),s\big)e^{-\lambda (s-t)}\,ds + g^\ell(y^\ell(T))e^{-\lambda (T-t)},
\end{equation}
and, the POD approximation for \eqref{ocp:full} reads as follows:
\begin{equation}
\label{ocp:red}
\min_{u\in U}  J_{x^\ell,t}^\ell(u)\quad\text{such that }\quad y^\ell(t) \mbox{ solves }\eqref{pod_sys}.
\end{equation}
Finally, we define the reduced value function $v^\ell(x^\ell,t)$ as
\begin{equation}
v^\ell(x^\ell,t):=\inf\limits_{u\in \mathcal{U}} J_{x^\ell,t}^\ell(u)
\label{value_fun_red}
\end{equation}
and the reduced HJB equation:
\begin{equation}\label{HJB-POD}
\left\{
\begin{array}{ll} 
&\dfrac{\partial v^\ell}{\partial s}(x^\ell,s) -\lambda v^\ell(x^\ell,s)+ \min\limits_{u\in U }\left\{L^\ell(x^\ell, u, s)+\nabla v^\ell(x^\ell, s) \cdot  f^\ell(x^\ell,u, s)\right\} = 0, \\
&v^\ell(x^\ell,T) = g^\ell( x^\ell), \qquad \qquad \qquad \qquad \qquad \qquad \qquad  (x^\ell,s) \in \mathbb{R}^\ell \times [t,T).
\end{array}
\right.
\end{equation}
Alternatively, one could further approximate the nonlinear term using DEIM and replace the dynamics \eqref{pod_sys}  with \eqref{pod_sysdeim} in \eqref{ocp:red}, providing an impressive acceleration of the algorithm as shown in Section \ref{sec:tests}.

\section{HJB-POD method on a tree structure}\label{sec:coupl}

In this section we explain, step by step, how to use model reduction techniques on a tree structure in order to obtain an efficient approximation of the value function and to deal with complex problems such as PDEs. 

\paragraph{\bf Computation of the snapshots} When applying POD for optimal control problems there is a major bottleneck: the choice of the control inputs to compute the snapshots.  Thus, we store the tree $\mathcal{T} = \cup_{n=0}^N \mathcal{T}^n$ for a chosen $\Delta t$ and discrete control set $U$. This set turns out to be a very good candidate for the snapshots matrix since it delivers all the possible trajectories we want to consider. To summarize the snapshots set is ${ Y} = \mathcal{T} = \cup_{n=0}^N \mathcal{T}^n$. In the numerical tests, we will use $\Delta t = 0.1$ and $2$ controls to compute the snapshots that, as shown in Section \ref{sec:tests}, will be sufficient to catch the main features of the controlled problem. 

\paragraph{\bf Computation of the basis functions} The computation of the basis $\Psi$ has been described in Section \ref{Sec:pod}. We are going to solve the following optimization problem:
\begin{equation}\label{pbmin2}
\min_{ {{\psi}}_1,\ldots,{{\psi}}_\ell\in\R^d} \sum_{j=1}^N \sum_{\underline{u}_j \subset U^j} \left|{ y}(t_j,\underline{u}_j)-\sum_{i=1}^\ell \langle { y}(t_j,\underline{u}_j),{{\psi}}_i\rangle{{\psi}}_i\right|^2\quad \mbox{such that }\langle {{\psi}}_i,{{\psi}}_j\rangle=\delta_{ij},
\end{equation}
where $\underline{u}_j=(u_1,\ldots,u_j) \subset U^j= U \times \ldots \times U$ and
$$
y(t_j,\underline{u}_j)=y_0 + \Delta t \sum_{k=0}^{j-1} f(y_{k},u_{k+1},t_k).
$$


In this context we have no restrictions on the choice of the number of basis $\ell$, since we will solve the HJB equation on a tree structure. In former works, e.g. \cite{KVX04,AFV17}, the authors were restricted to choose $\ell \approx 4$ to have a feasible reduction of the HJB equation. Here, the dimension of the state variable is not a major issue. On the other hand, the pruning strategy will turn out to be crucial for the feasibility of the problem. 

It is well-known that the error in \eqref{pbmin} is given by the sum of the singular values neglected. We recall that we will chose $\ell$ such that $\mathcal{E}(\ell)\approx 0.999,$ with $\mathcal{E}(\ell)$ defined in \eqref{POD_cri}.

\paragraph{\bf Construction of the reduced tree}  
Having computed the POD basis, we build a new tree which might consider a different $\Delta t$ and/or a finer control space with respect to  the snapshots set. We will denote the projected tree as $\mathcal{T}^\ell$ with its generic $n-$th level given by:
$$\mathcal{T}^{n,\ell} = \{ \zeta^{n-1,\ell}_i + \Delta t f^\ell(\zeta^{n-1,\ell}_i, u_j,t_{n-1}),\; j = 1,\ldots, M,\; i = 1,\ldots, M^{n-1} \},$$ 
where the reduction of the nonlinear term $f^\ell$ can be done via POD or POD-DEIM as in \eqref{f_DEIM}. The first level of the tree is clearly given by the projection of the initial condition, i.e. $\mathcal{T}^{0,\ell}= \Psi^T x$. Then, the procedure follows the full dimensional case, but with the projected dynamics. We will show how this approach speeds up the method keeping high accuracy.
Even if we have reduced the dimension of the problem, the cardinality  of the tree $\mathcal{T}^{n,\ell}$ depends on the number of the discrete controls and the time step chosen as in the high-dimensional case. It is clear that each resolution of the PDE will be faster, but it is still necessary to apply a pruning rule which reads:

\begin{equation}\label{tol_cri_pod}
\begin{array}{cc}
\Vert \zeta^{n,\ell}_i-\zeta^{n,\ell}_j \Vert \le \ep,
\mbox{ for  }i\ne j \mbox{ and } n = 0,\ldots, \overline{N}.
\end{array}
\end{equation}

As proposed in \cite{AFS18}, the evaluation of \eqref{tol_cri_pod} can be computed in a more efficient way, considering the most variable components by the principal component analysis. This technique is incorporated in our algorithm, since we have already computed the POD basis and the most variable component turns out to be the first one $y_1^\ell$. It will be sufficient to reorder the nodes according to their first components to accelerate the pruning criteria.

\paragraph{\bf Approximation of the reduced value function} 
The numerical reduced value function $V^\ell(x^\ell,t)$ will be computed on the tree nodes in space as 
\begin{equation}\label{num:vf}
V^\ell(x^\ell,t_n)=V^{n,\ell}(x^\ell), \quad \forall x^\ell \in \mathcal{T}^{n,\ell}.
\end{equation}
 Then, the computation of the reduced value function follows directly from the DPP. Defined the grid $\mathcal{T}^{n,\ell}=\{\zeta^{n,\ell}_j\}_{j=1}^{M^n}$ for $n=0,\ldots, \overline{N}$, we can write a time discretization for \eqref{HJB} as follows: 

\begin{equation}
\begin{cases}
V^{n,\ell}(\zeta^{n,\ell}_i)= \min\limits_{u\in U} \{e^{-\lambda \Delta t} V^{n+1,\ell}(\zeta^{n,\ell}_i+\Delta t f^\ell (\zeta^{n,\ell}_i,u,t_n)) +\Delta t \, L^\ell(\zeta^{n,\ell}_i,u,t_n) \}, \\
 \qquad \qquad\qquad \qquad\qquad  \zeta^{n,\ell}_i \in \mathcal{T}^{n,\ell}\,, n = \overline{N}-1,\ldots, 0, \\
V^{\overline{N},\ell}(\zeta^{\overline{N},\ell}_i)= g^\ell(\zeta_i^{\overline{N},\ell}), \qquad\qquad \qquad\qquad \qquad\qquad   \zeta_i^{\overline{N},\ell} \in \mathcal{T}^{\overline{N},\ell}.
\end{cases}
\label{HJBt-POD}
\end{equation}

\paragraph{\bf Computation of the feedback control} 

The computation of the feedback control strongly relies on the fact we deal with a discrete control set $U$. Indeed, when we compute the reduced value function, we store the control indices corresponding to the $\argmin$ in \eqref{HJBt-POD}. The optimal trajectory is than obtained by following the path of the tree with the controls chosen such that 	

\begin{equation} \label{feed:tree-POD}
u^{n,\ell}_{*}:=\argmin\limits_{u\in U} \left\{ e^{-\lambda \Delta t}V^{n+1,\ell}(\zeta^{n,\ell}_*+\Delta t f^\ell(\zeta^{n,\ell}_*,u,t_n)) +\Delta t \, L^\ell(\zeta^{n,\ell}_*,u,t_n) \right\},
\end{equation}
\begin{equation*} 
\zeta^{n+1,\ell}_* \in \mathcal{T}^{n+1,\ell} \; s.t. \; \zeta^{n,\ell}_* \rightarrow^{u_{n}^{*}} \zeta^{n+1,\ell}_*,
\end{equation*}
for $n=0,\ldots, \overline{N}-1$, where the symbol $\rightarrow^u$ stands for the connection of two nodes by the dynamics corresponding to the control $u$.

Once the control $u_*^{n,\ell}$ has been computed, we plug it into the high dimensional problem \eqref{ode} and compute the optimal trajectory.
  
\section{Error estimates for the HJB-POD method on a TSA} \label{sec:err}
In this section we derive an error estimate for the HJB-POD approximation \eqref{HJBt-POD} on a tree structure. In what follows, we assume that the functions $f, L, g$ are bounded:
 \begin{align}
 \begin{aligned}\label{Mf}
|f(x,u,s)|& \le M_f,\quad |L(x,u,s)| \le M_L,\quad |g(x)| \le M_g, \cr
&\forall\, x \in \mathbb{R}^d, u \in U \subset \mathbb{R}^m, s \in [t,T], 
\end{aligned}
\end{align}
the functions $f$ and $L$ are Lipschitz-continuous with respect to the first variable
\begin{align}
\begin{aligned}\label{Lf}
&|f(x,u,s)-f(y,u,s)| \le L_f |x-y|, \quad |L(x,u,s)-L(y,u,s)| \le L_L |x-y|,\cr
&\qquad\qquad\qquad\qquad\forall \, x,y \in \mathbb{R}^d, u \in U \subset \mathbb{R}^m, s \in [t,T], 
\end{aligned}
\end{align}
and the cost $g$ is also Lipschitz-continuous:
\begin{equation}
|g(x)-g(y)| \le L_g |x-y|, \quad \forall x,y \in \mathbb{R}^d.
\label{Lg}
\end{equation}
Furthermore, let us assume that the functions $L$ and $g$ are semiconcave
\begin{align}
\begin{aligned}\label{CL}
&L(x+z,u,t+\tau)-2L(x,u,t)+L(x-z,u,t-\tau) \le C_L (|z|^2+ \tau^2),\cr
&g(x+z)-2g(x)+g(x-z) \le C_g |z|^2, \qquad \forall x,z \in \mathbb{R}^d, u \in U, t,\tau \ge 0,
\end{aligned}
\end{align}
and assume that $f$ verifies the following inequality:
\begin{align}
\begin{aligned}\label{Cf}
&|f(x+z,u,t+\tau)-2f(x,u,t)+f(x-z,u,t-\tau) | \le C_f (|z|^2+ \tau^2),\cr
 &\qquad\qquad\qquad \forall u \in U, \;\forall x,z \in \mathbb{R}^d,\, \forall t,\tau \ge 0.
\end{aligned}
\end{align}
We also introduce the continuous-time extension of the DDP
\begin{align}\label{HJBt3}
\begin{aligned}
&V(x,s) = \min\limits_{u\in U} \{e^{-\lambda (t_{n+1}-s)} V(x+(t_{n+1}-s) f(x,u,s), t_{n+1}) + (t_{n+1}-s) \, L(x,u,s) \}, \\
&V(x,T) = g(x), \hspace{6cm} x \in \mathbb{R}^d, s \in [t_n,t_{n+1}),
\end{aligned}
\end{align}
and the POD version for the continuous-time extension \eqref{HJBt3} which reads:
\begin{align}\label{HJBt4}
\begin{aligned}
&V^\ell(x^\ell,s) = \min\limits_{u\in U} \{e^{-\lambda (t_{n+1}-s)} V^\ell(x^\ell+(t_{n+1}-s) f^\ell(x^\ell,u,s), t_{n+1}) +\\
&\qquad\qquad\qquad\qquad\qquad\quad\qquad\qquad\qquad\qquad\qquad +(t_{n+1}-s) \, L^\ell(x^\ell,u,s) \}, \\
&V^\ell(x^\ell,T) = g^\ell(x^\ell), \hspace{6cm} x^\ell \in \mathbb{R}^\ell, s \in [t_n,t_{n+1}).
\end{aligned}
\end{align}

Given the exact solution $v(x,s)$ and its POD discrete approximation $V^\ell(x^\ell,s)$, we prove the following theorem which provides an error estimate for the proposed method.

\begin{teo}
Let us assume $\eqref{Mf}$-$\eqref{Cf}$ hold true, then there exists a constant $C(T)$ such that
\begin{equation}
\sup_{s \in [t,T]}|v(x,s)- V^\ell(x^\ell,s) | \le C(T) \left(  \left( \sum_{i \ge \ell +1} \sigma_i^2 \right)^{1/2}+\Delta t\right)
\label{err_est}
\end{equation}
where the $\{\sigma_i\}_{i=1}^{\min\{N+1,d\}}$ are the singular values of the snapshots matrix.
\label{teoHJBPOD}
\end{teo}

\begin{proof}
We observe that, by triangular inequality, the approximation error can be decomposed in two parts:
\begin{equation}\label{triang}
|v(x,s)-V^\ell(x^\ell,s) | \le |v(x,s)-V(x,s)| +|V(x,s)- V^\ell(x^\ell,s) |.
\end{equation}
An error estimate for the first term has been already obtained in \cite{SAF18}:
\begin{equation}
\sup_{(x,s)\in \mathbb{R}^d \times [0,T]} \left|V(x,s)-v(x,s)\right| \le \widehat{C}(T) \Delta t.
\label{2est}
\end{equation}
Let us focus on the second term of the right hand side of \eqref{triang}. Without loss of generality, we consider $\lambda=0$. 
For $s=T$, the estimate follows directly by the assumptions on $g$. Considering $x \in \mathbb{R}^d$ and $s \in [t_n,t_{n+1})$, we can write
$$
V(x,s)-V^\ell(x^\ell,s) \le 
$$
$$
 V(x_{n+1},t_{n+1}) -V^\ell(x^{\ell}_{n+1},t_{n+1})+ (t_{n+1}-s)\left(L(x,u^n_*,s) - L^\ell(x^\ell,u^n_*,s)\right) \le
 $$
 \begin{equation}
  V(x_{n+1},t_{n+1}) -V^\ell(x^{\ell}_{n+1},t_{n+1}) + (t_{n+1}-s)\, L_L |x-\Psi x^\ell|,
 \label{boh}
 \end{equation}
 
  where $u^n_*, x_{n+1}$ and $x^\ell_{n+1}$ are defined as
$$
u^n_*= \argmin_{u\in U }\left\{V^\ell(x^\ell+ (t_{n+1}-s) f^\ell(x^\ell,u,s),t_{n+1})+ (t_{n+1}-s) \, L^\ell(x^\ell, u,s)\right\},
$$
$$
x_{n+1}=x+ (t_{n+1}-s) f(x,u^n_*,s),\qquad  x^\ell_{n+1}=x^\ell+ (t_{n+1}-s) f^\ell(x^\ell,u^n_*,s).
$$

 We define the trajectory path and its POD approximation respectively as 
$$
x_{m}:= x+ \sum_{k=n}^{m-1} \alpha_k f(x_k, u_*^k,\bar{t}_k), \quad
x^\ell_m:=x^\ell +\sum_{k=n}^{m-1} \alpha_k f^\ell(x^\ell_k, u_*^k,\bar{t}_k),
$$
where
$$
\alpha_k= \begin{cases}
		 t_{n+1}-s & k=n\\
		 \Delta t & k \ge n+1
		 \end{cases} ,
		 \quad
\bar{t}_k= \begin{cases}
		 s & k=n\\
		 t_k & k \ge n+1
		 \end{cases} ,
$$
$$
u^{k}_*= \argmin_{u\in U }\left\{V^\ell\left(x^\ell_k+ \alpha_k  f^\ell( x^\ell_k,u,\bar{t}_k),t_{k+1}\right)+\alpha_k L^\ell(x^\ell_k,u,\bar{t}_{k}) \right\}, k \ge n,
$$
with $x_n=x$ and $x_n^\ell=x^\ell$. Then, iterating \eqref{boh} we obtain
\begin{equation}
V(x,s)-V^\ell(x^\ell,s) \le  L_L  \sum_{m=n}^{\overline{N}-1} \alpha_m |x_{m}-\Psi x_{m}^\ell|+ L_g|x_{\overline{N}}-\Psi x_{\overline{N}}^\ell|.
\label{est}
\end{equation}

Defining
$$
\eta_m= \begin{cases}
		L_L \alpha_m & m\in \{n, \ldots \overline{N}-1 \}\\
		 L_g & m=\overline{N}
		 \end{cases},
$$

\noindent
we can write
$$
V(x,s)-V^\ell(x^\ell,s) \le  \sum_{m=n}^{\overline{N}} \eta_m |x_{m}-\Psi x_{m}^\ell|.
$$
By triangular inequality and Cauchy-Schwarz inequality, we can write
$$
V(x,s)-V^\ell(x^\ell,s) \le   \sum_{m=n}^{\overline{N}} \eta_m \left( |x_{m}-\mathcal{P}^\ell x_{m}|+|\mathcal{P}^\ell x_{m}-\Psi x_{m}^\ell|\right) \le
$$
\begin{equation}
\left(\sum_{m=n}^{\overline{N}} \eta_m^2 \right)^{1/2} \left(   \left( \sum_{m=n}^{\overline{N}} |x_{m}-\mathcal{P}^\ell x_{m}|^2\right)^{1/2}+\left( \sum_{m=n}^{\overline{N}} |\mathcal{P}^\ell x_{m}-\Psi x_{m}^\ell|^2 \right)^{1/2}\right) ,
\label{Vell}
\end{equation}
where $\mathcal{P}^\ell=\Psi^T \Psi$ is a projection operator. Since $\{x_m\}_m \subset \mathcal{T}$, by the definition of POD basis we get

\begin{equation}
\left( \sum_{m=n}^{\overline{N}} |x_{m}-\mathcal{P}^\ell x_{m}|^2\right)^{1/2} \le \left( \sum_{i \ge \ell +1} \sigma_i^2 \right)^{1/2}.
\label{err}
\end{equation}

Let us denote by $Err(\ell)=\left( \sum_{i \ge \ell +1} \sigma_i^2 \right)^{1/2}$ the error related to the orthogonal projection onto $V^\ell$.

Let us focus now on the generic term $|\mathcal{P}^\ell x_{m}-\Psi x_{m}^\ell|$:
$$
|\mathcal{P}^\ell x_{m}-\Psi x_{m}^\ell| \le  \sum_{k=n}^{m-1} \alpha_k  \Vert \mathcal{P}^\ell \Vert_2 | f(x_k, u_*^k,\bar{t}_k)- f(\Psi x^\ell_k, u_*^k,\bar{t}_k) | \le
$$
$$
L_f \Vert \mathcal{P}^\ell\Vert_2  \sum_{k=n}^{m-1} \alpha_k   |x_k-\Psi x^\ell_k | \le L_f \Vert \mathcal{P}^\ell \Vert_2  \sum_{k=n}^{m-1} \alpha_k \left(  |x_k-\mathcal{P}^{\ell} x_k | + |\mathcal{P}^\ell x_k - \Psi x_k^\ell| \right).
$$
By the discrete Gr\"onwall's lemma and noticing that $\Vert \mathcal{P}^\ell \Vert_2 =1$ , we get

\begin{equation*}
|\mathcal{P}^\ell x_{m}-\Psi x_{m}^\ell| \le L_f \sum_{k=n}^{m-1} \alpha_k   |x_k-\mathcal{P}^{\ell} x_k | e^{L_f  (t_m-s) } ,
\label{Pell}
\end{equation*}
and since $\alpha_k \le \Delta t$ $\forall k$, we obtain

\begin{equation}
 \left( \sum_{m=n}^{\overline{N}}  |\mathcal{P}^\ell x_{m}-\Psi x_{m}^\ell|^2 \right)^{1/2} \le \sqrt{T-s} L_f  e^{L_f   (T-s)}  Err(\ell).
\label{Pell2}
\end{equation}

Plugging  \eqref{err} and \eqref{Pell2} into \eqref{Vell} we get
$$
V(x,s)-V^\ell(x^\ell,s) \le Err(\ell) \left(\sum_{m=n}^{\overline{N}} \eta_m^2 \right)^{1/2} \left(  \sqrt{T}  L_f  e^{L_f   T}+1\right) .
$$
Finally, noticing that 
$$
 \sum_{m=n}^{\overline{N}} \eta_m^2  \le (T L_L)^2 + L_g^2,
$$
we obtain
$$
V(x,s)-V^\ell(x^\ell,s) \le C_1(T) Err(\ell),
$$
where
$$
C_1(T)= \left( (T L_L)^2+ L_g^2\right)^{1/2}   \left(  \sqrt{T} L_f  e^{L_f   T}+1\right) .
$$
Analogously, it is possible to obtain the same estimate for $V^\ell(x^\ell,s)-V(x,s)$ and, defining $C(T)=\max \{\widehat{C}(T), C_1(T) \}$, we get the desired result.

\end{proof}

\begin{rmk}

The error estimate presented in Theorem \ref{teoHJBPOD} depends strongly on the initial condition, since the POD reduction is based on the tree generated by the starting point $x$. We can extend the error estimate to other initial conditions if we enlarge the snapshots set with these new data and their evolutions up to the final time $T$.

\end{rmk}

\section{Numerical Tests}\label{sec:tests}

In this section we apply our proposed algorithm to show the effectiveness of the method with two test cases. In the first we deal with a parabolic PDE with a polynomial nonlinear term, which is usually not a trivial task when applying open-loop control tools. The second test concerns the bilinear control of the viscous Burgers' equation. 

In order to obtain the PDEs in the form \eqref{ode}, we use a Finite Difference scheme and we integrate in time using an implicit Euler scheme coupled with the Newton's method with tolerance equal to $10^{-4}$. We will denote by $U_n$ the discretized set of $U$ with $n$ equi-distributed controls.

The numerical simulations reported in this paper are performed on a MacBook Pro with 1CPU Intel Core i7, $2.6$ GHz and 16GB RAM. The codes are written in Matlab R2018b.

\subsection{Test 1: Nonlinear reaction diffusion equation}\label{sec:61}
In the first example we consider the following bidimensional PDE with polynomial nonlinearity and homogeneous Neumann boundary conditions 
\begin{equation}
\begin{cases}
\partial_s y= \sigma \Delta y +\mu\left(y^2-y^3 \right)+  y_0(x)u(s) & (x,s) \in \Omega \times [0,T],
\\
\partial_n y(x,s) =0 & (x,s) \in \partial \Omega \times [0,T], \\
y(x,0)=y_0(x) & x \in \Omega,
\end{cases}
\label{pde1}
\end{equation}
where $y:\Omega\times [0,T]\rightarrow \R$, the control $u(t)$ is taken in the admissible set $\mathcal{U}=\{u:[0,T]\rightarrow [-2,0] \}$ and $\Omega=[0,1]^2$.
In \eqref{pde1} we consider: $
T=1, \sigma=0.1, \mu=5$ and $y_0(x_1,x_2)=sin(\pi x_1)sin(\pi x_2).
$
  We discretize the space domain $\Omega$ in $31$ points in each direction, obtaining a discrete domain with $d=961$ points. As shown in Figure \ref{fig1}, the solution of the uncontrolled equation \eqref{pde1} (i.e. $u(t)\equiv 0$) converges asymptotically to the stable equilibrium $\overline{y}_1(x) =1$.  

\begin{figure}[htbp]	
\centering
	\includegraphics[scale=0.3]{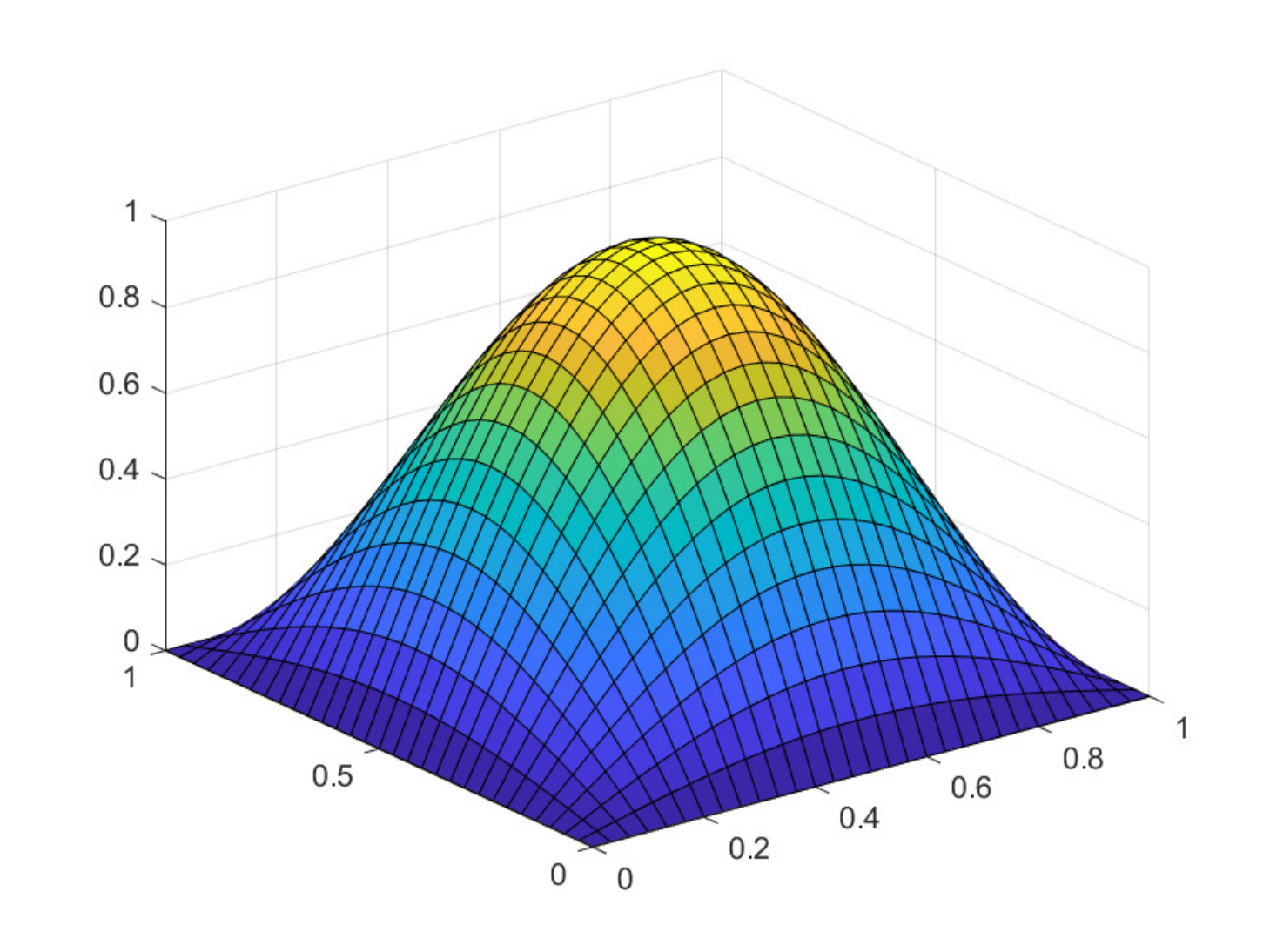}	
	\includegraphics[scale = 0.3]{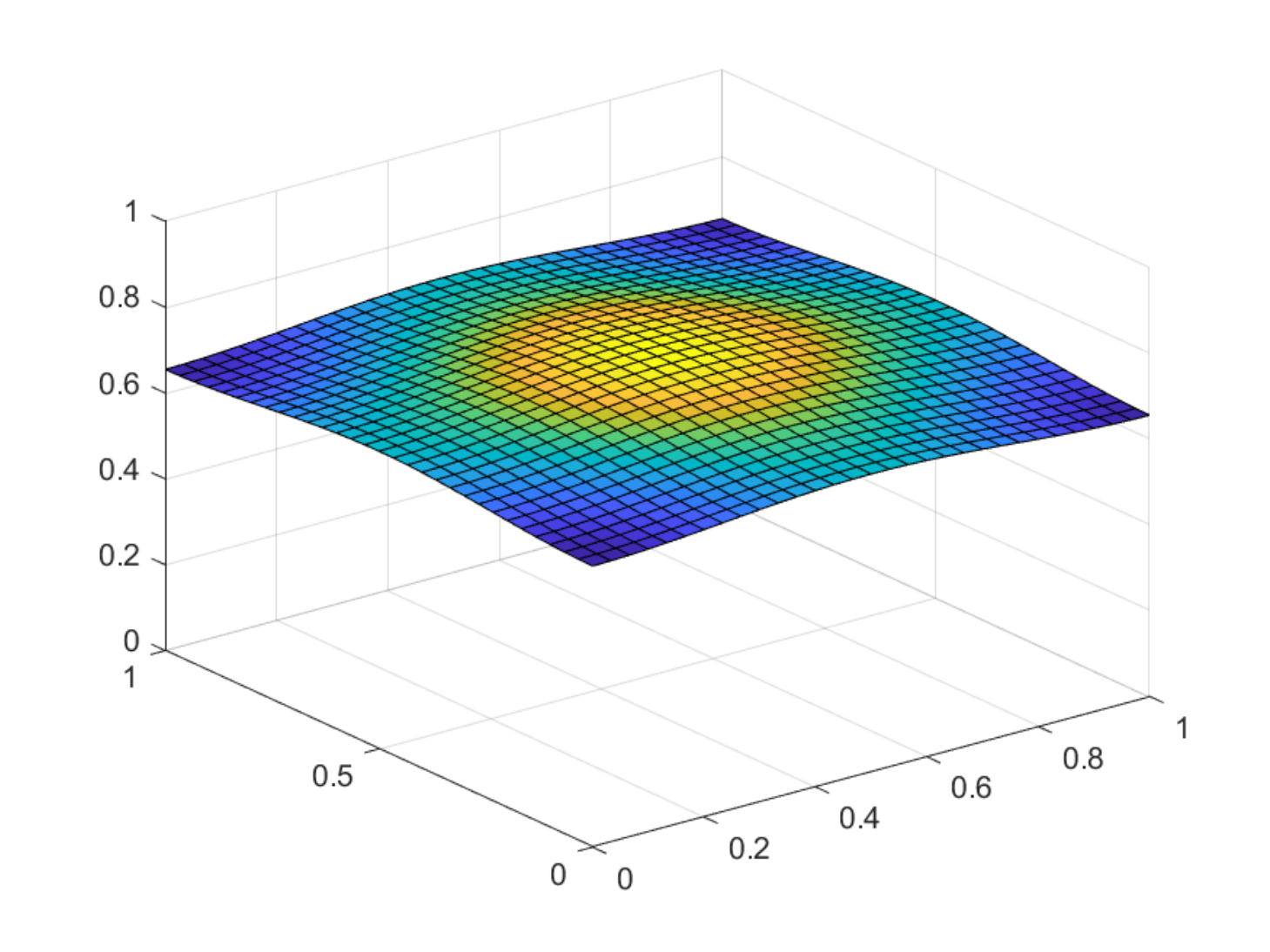}	
	\includegraphics[scale= 0.3]{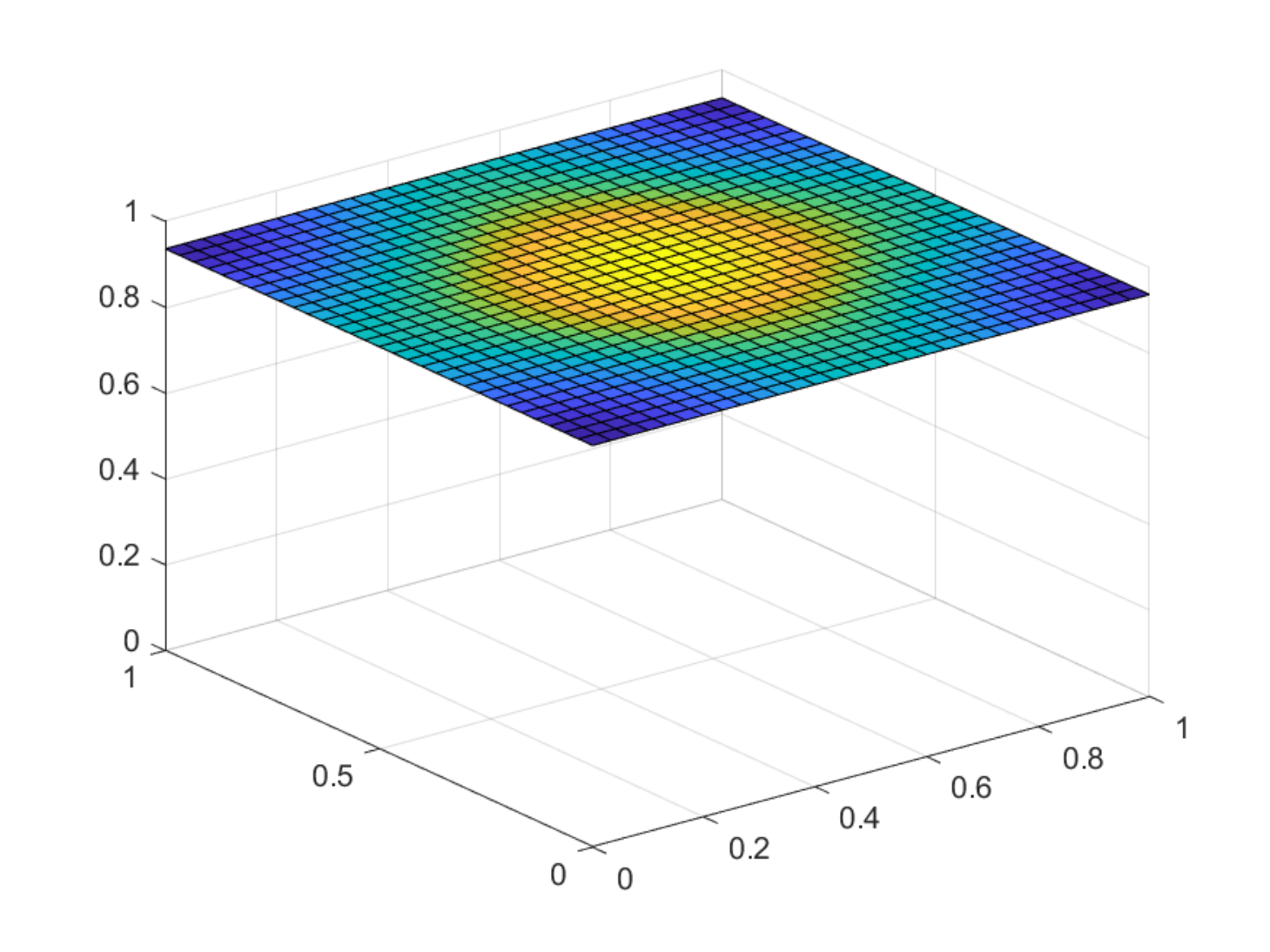}		
\caption{Test 1: Uncontrolled solution for equation \eqref{pde1} for time $t=\{0, 0.5, 1\}$ (from left to right).}
\label{fig1}	
\end{figure}

Our aim is to steer the solution to the unstable equilibrium $\overline{y}_2(x) =0$. For this reason, we introduce the following cost functional
\begin{equation}\label{cost_test}
J_{y_0,t}(u) = \int_t^T \left(  \int_{\Omega} |y(x,s)|^2 dx + \dfrac{1}{100}  |u(s)|^2 \right) ds + \int_{\Omega} |y(x,T)|^2  dx.
\end{equation}

\paragraph{Case 1: Full TSA}

We first consider the results using the TSA without model order reduction. In Figure \ref{fig3:heat} we report the optimal trajectory obtained using the full tree structure algorithm with $2$ controls and $\Delta t=0.1$.
\begin{figure}[htbp]	
\centering
	\includegraphics[scale=0.3]{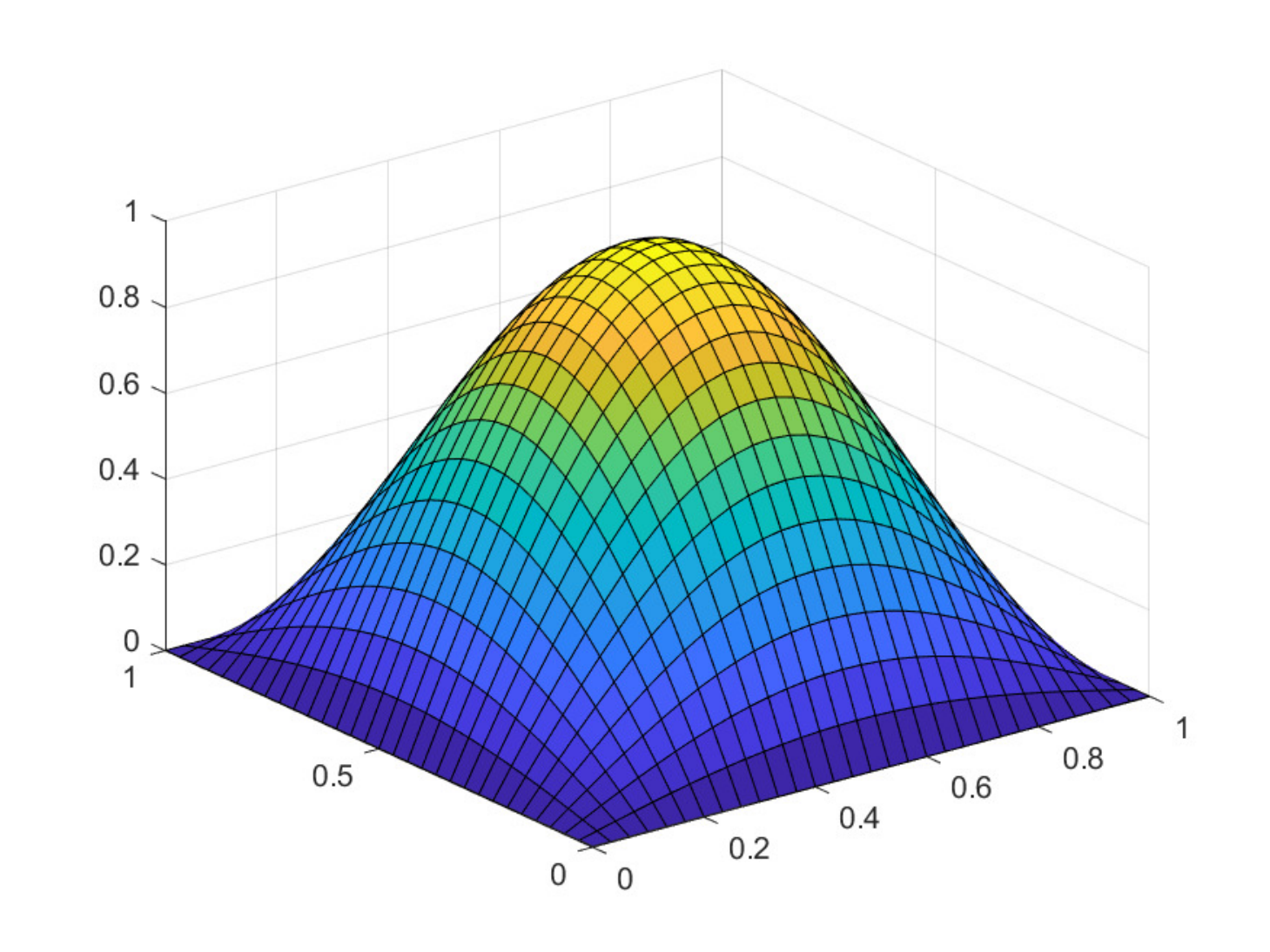}	
	\includegraphics[scale = 0.3]{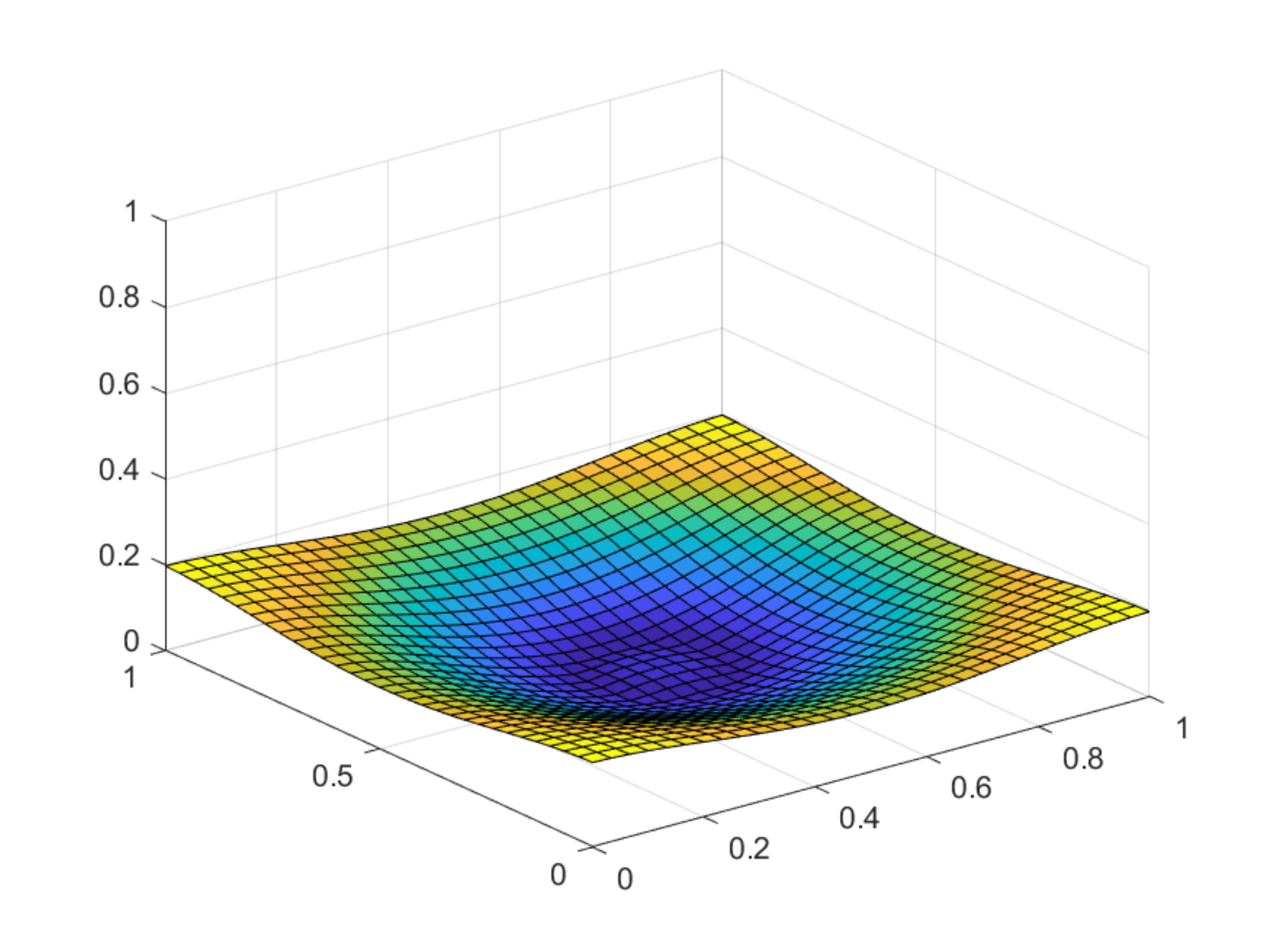}	
	\includegraphics[scale= 0.3]{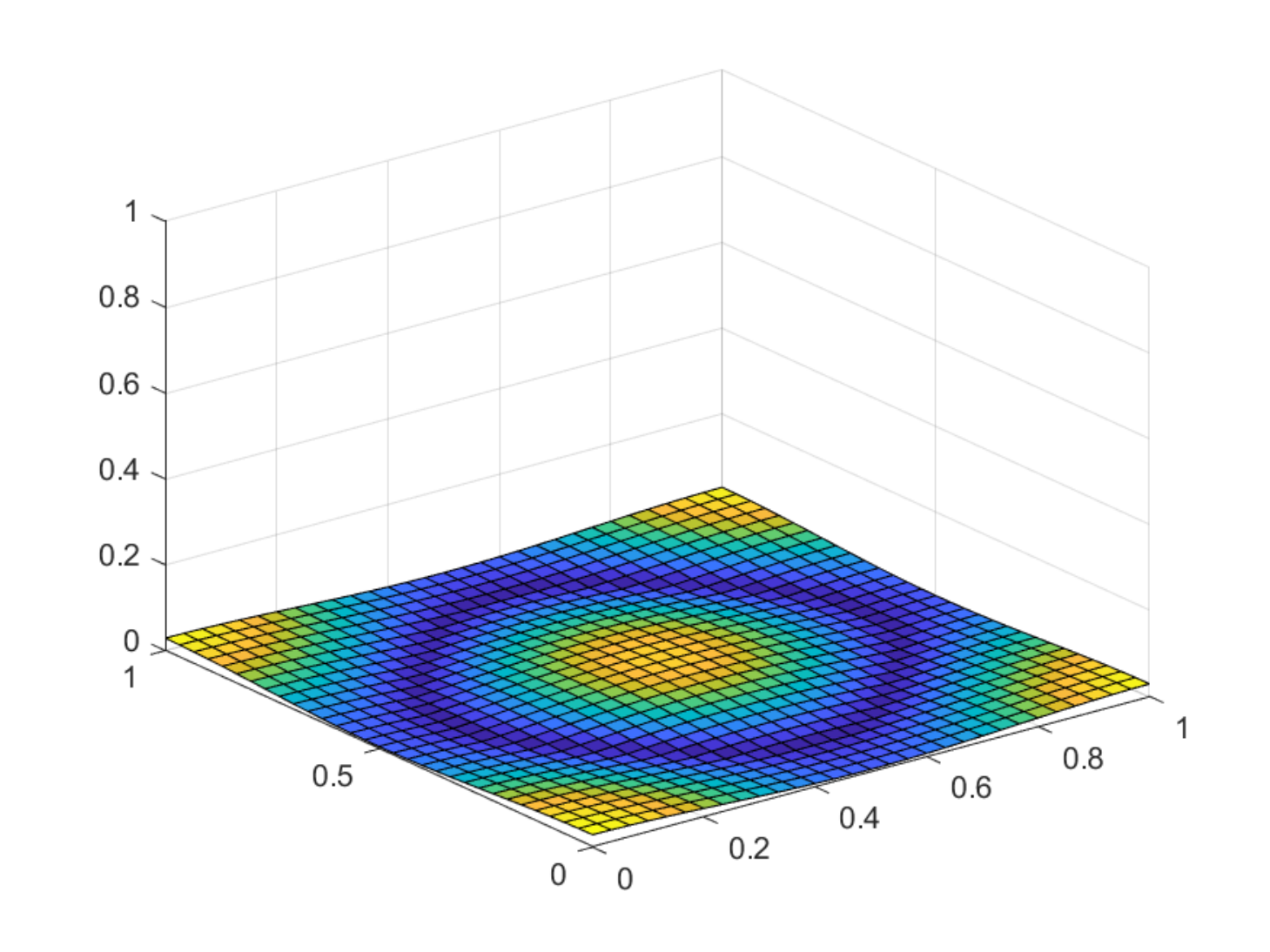}		
\caption{Test 1: Controlled solution with TSA for equation \eqref{pde1} with full tree for time $t=\{0, 0.5, 1\}$ (from left to right) with $U_2$.}	
\label{fig3:heat}
\end{figure}
As one can see, we steer the solution to the unstable equilibrium using $U_2=\{-2,0\}$ as discrete control set. For the given tolerance $\ep =\Delta t^2=0.01$, the cardinality of the pruned tree with $3$ controls is $84354$, whereas without is $88573$.

In the left panel of Figure \ref{fig3}, we show the control policy obtained with $2$, $3$ and $4$ discrete controls. In the right panel we show the behaviour of the cost functional, and it is easy to check that the optimal trajectories are very similar. An analysis of the CPU time is provided in Table \ref{tab1:cpu} and discussed below.

\begin{figure}[htbp]	
\centering
	\includegraphics[scale=0.4]{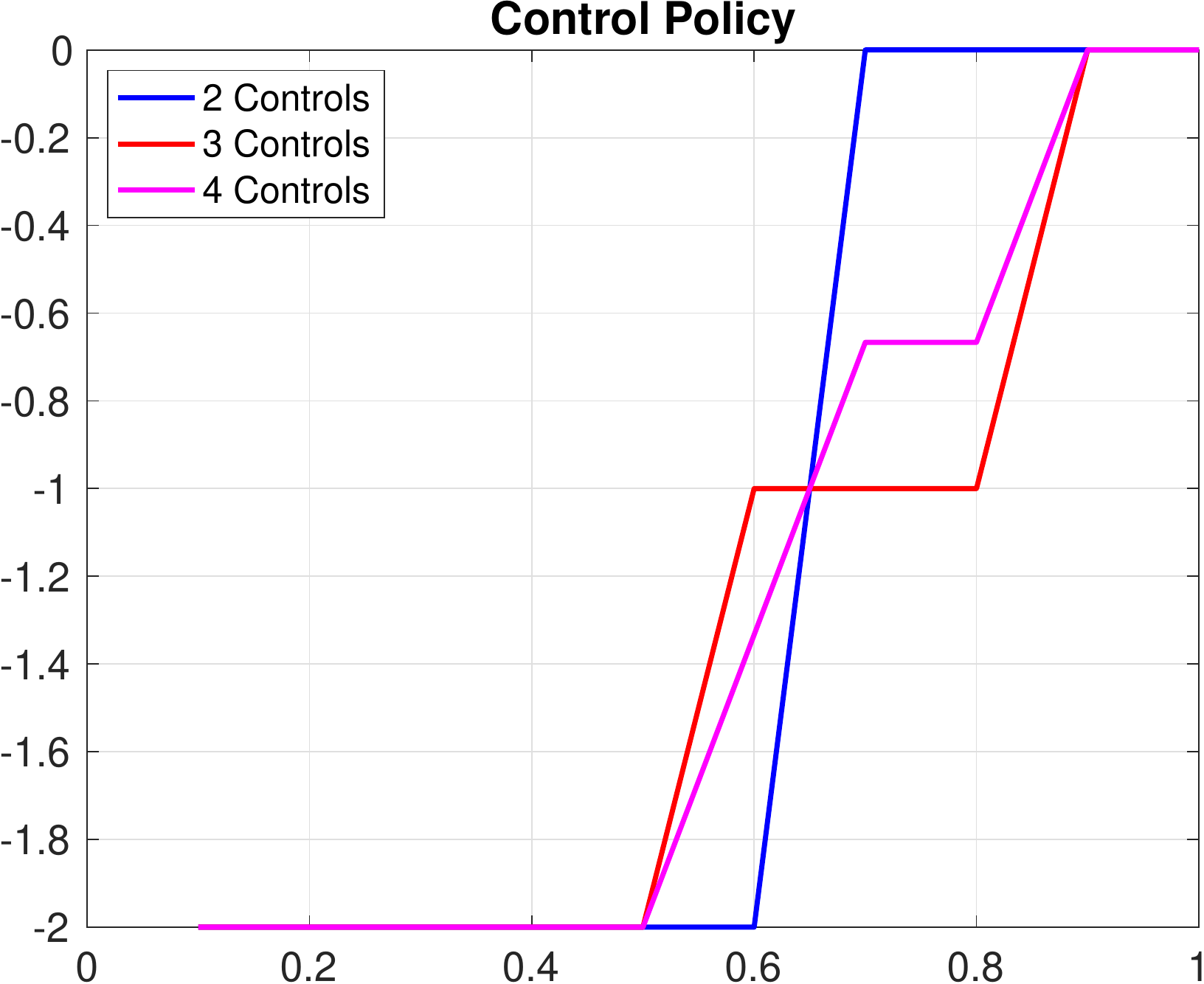}	
	\includegraphics[scale=0.4]{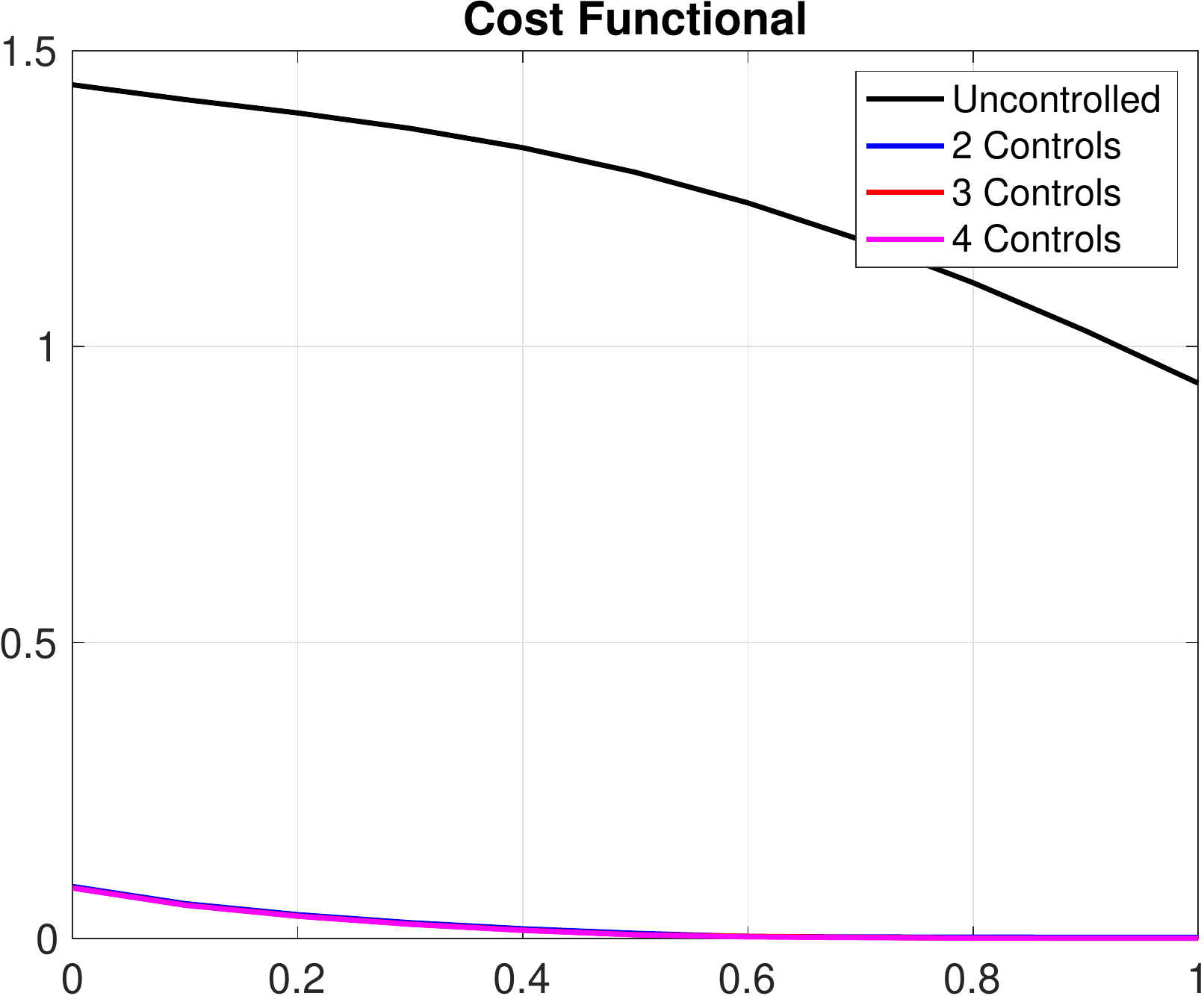}	
\caption{Test 1: Control policy (left) and cost functional (right) for $U_2$, $U_3$ and $U_4$.}	
\label{fig3}
\end{figure}

\paragraph{Case 2: TSA with POD}

The computation of the full TSA is already expensive with only $3$ controls. For this reason, we replace the dynamics with its reduced order modeling. 
Then, we set the number of POD basis $\ell=6$ such that $\mathcal{E}(\ell)=0.999$. Similarly, we consider $6$ DEIM basis for the nonlinear term. In what follows, whenever we will talk about POD, we will refer to POD-DEIM approach.

 The snapshots matrix $Y$ is computed with a full TSA using the discrete control space $U_2$ and $\Delta t=0.1$. In the online stage we considered again $\Delta t=0.1$, a pruning criteria with $\ep =\Delta t^2$ and different discrete controls. 

In Figure \ref{fig:err}, we present the relative error with the euclidean norm between the model order reduction approximation and the full tree for $\Delta t = 0.1$ and $3$ controls. The snapshots in this example were computed with $\Delta t = 0.1$ and $2$ controls. We can observe that the approximation of the tree is rather accurate as the number of the POD basis $\ell$ increases.
\begin{figure}
\centering
\includegraphics[scale=0.4]{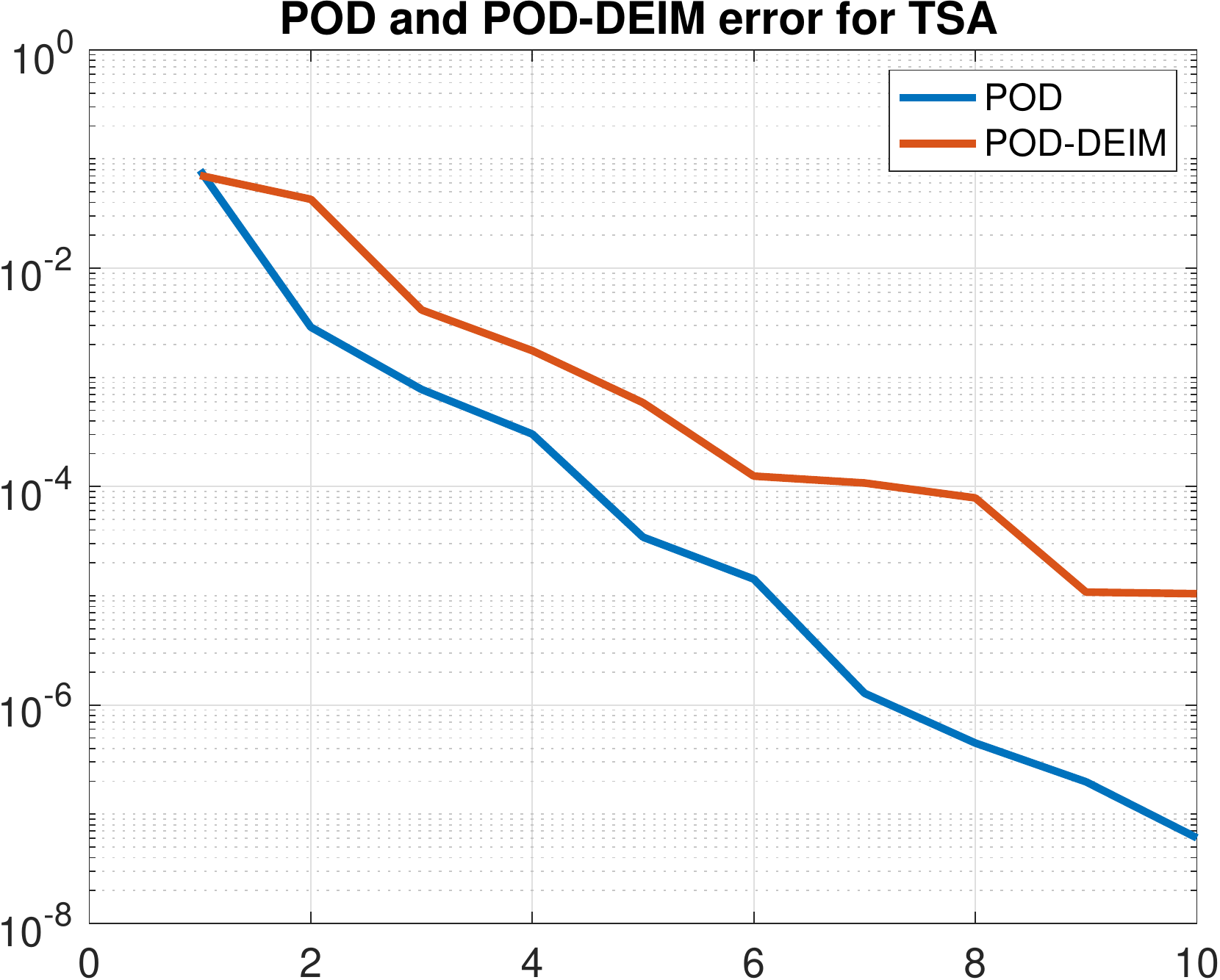}
\caption{\second{Relative euclidean error for the POD and POD-DEIM approximation of the tree. The snapshots are computed with $\Delta t = 0.1$ and $2$ controls, whereas the online stage refers to $\Delta t=0.1$ and $3$ controls. The $x$-axis refers to the number of POD (POD-DEIM) basis}.}
\label{fig:err}
\end{figure}

In the left panel of Figure \ref{fig4} we show the optimal policy with a number of controls varying from two to five. As one can see comparing the left panels of Figure \ref{fig3} and Figure \ref{fig4}, there is no difference in terms of optimal control between the high dimensional case discretized with Finite Difference and the low dimensional case obtained via POD. We remind that the optimal trajectory is obtained plugging the suboptimal control $u^\ell_*$ into the high dimensional model. Finally, in the right panel of Figure \ref{fig4} we show a zoom of the cost functional $J_{y_0,0}$ and it is possible to see the improvement obtained using more controls.


\begin{figure}[htbp]	
\centering
	\includegraphics[scale=0.4]{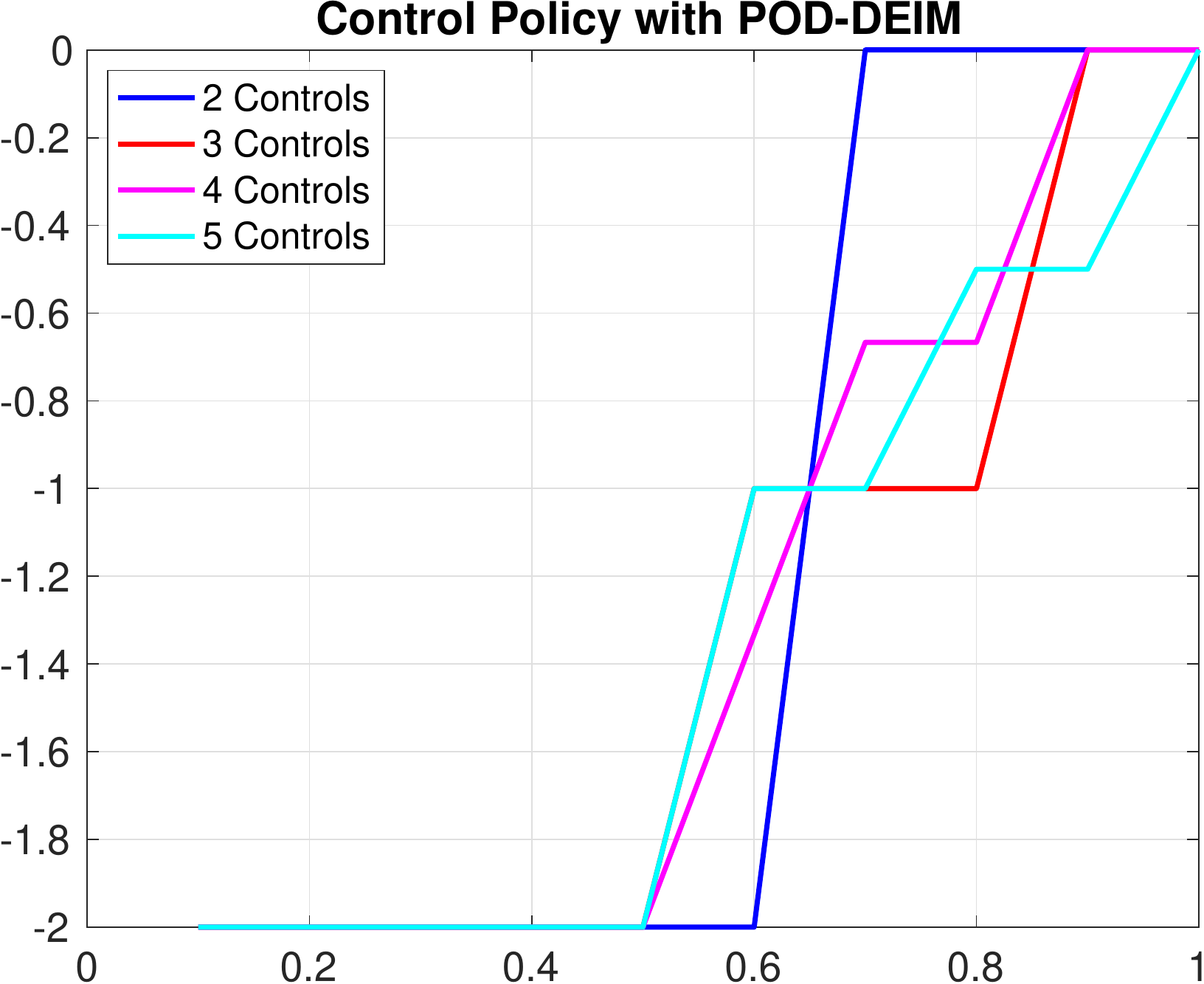}	
	\includegraphics[scale=0.4]{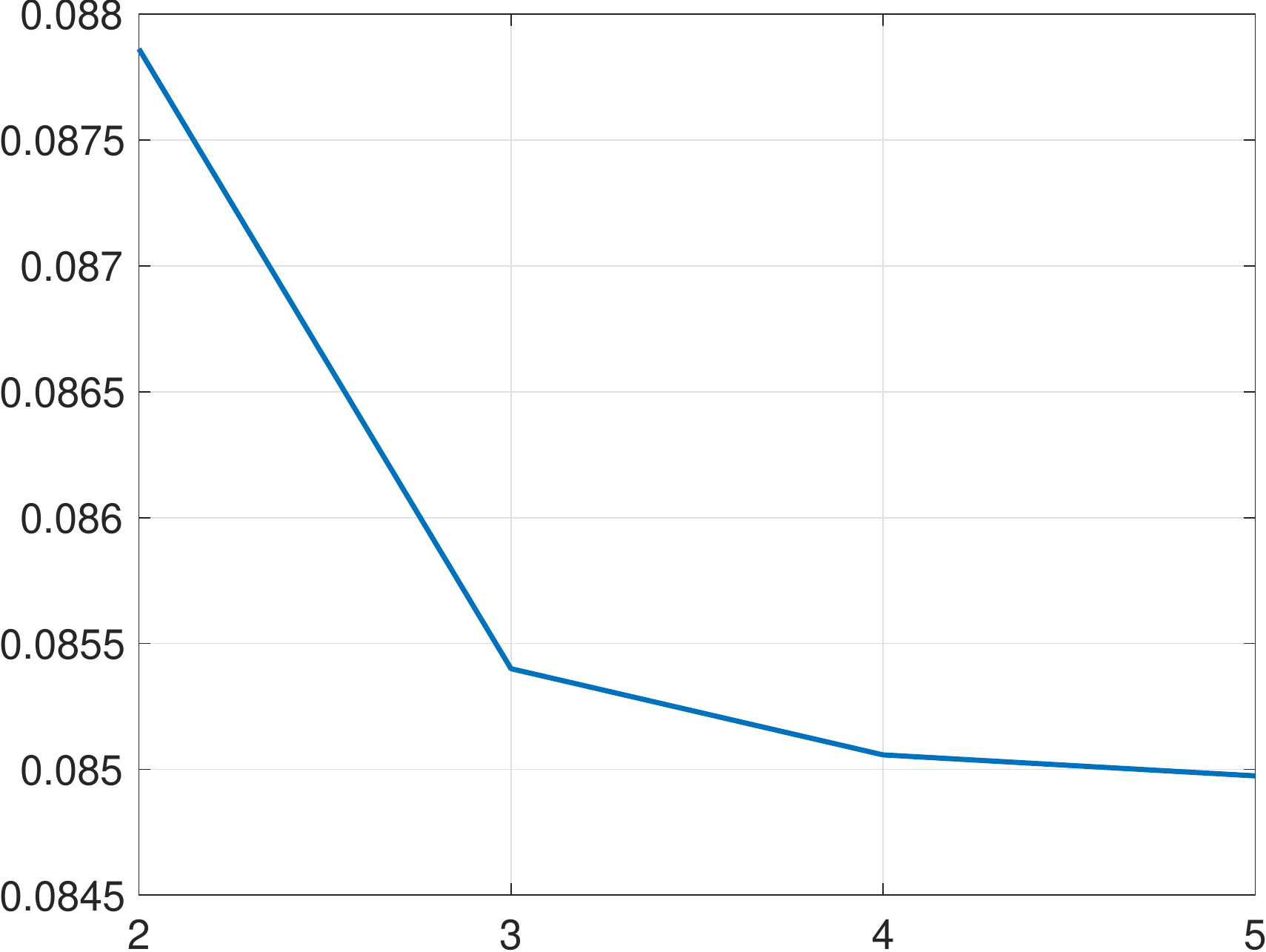}
\caption{Test 1: Optimal policy (left) and $J_{y_0,0}$ (right) for $U_n$ with $n=\{2,3,4,5\}$.}	
\label{fig4}
\end{figure}

The CPU time, expressed in seconds, is shown in Table \ref{tab1:cpu}. The online phase of the TSA-POD is always faster than the full TSA. We tried to compute the full TSA with $5$ controls and we stopped the computation after $4$ days. If we also consider the amount of time to compute the snapshots, the offline phase, using the TSA with 2 controls and then running online, e.g. the TSA-POD with 3 controls, we get a speed up of factor $10$ with respect to the full problem, having the same approximation.
\begin{table}[H]
\centering
\begin{tabular}{c c c c c }
& $U_2$ & $U_3$ & $U_4$ & $U_5$\\
\hline
TSA & $5.8312s$  & $241.5773$s & $3845.77$s & $>4$ days\\
TSA-POD & $0.5157$s &  $19.7969$s &$432.0990$s & $1.0871e+04$s\\
\hline
\end{tabular}
\caption{CPU time of the TSA and  the TSA-POD with a different number of controls and pruning criteria $\ep = 0.01$.}
\label{tab1:cpu}
\end{table}

%

\begin{rmk}
The offline stage of the proposed method is clearly expensive due to the cardinality of the tree. We have also tried to compute snapshots for some given control input setting, e.g. $u(t)\equiv \overline{u}$, with $\overline{u} \in \{-2,-1,0\}$. In this setting we are able to achieve the same results shown in the section, improving the computational performances of the method in the offline phase.
\end{rmk}

\begin{rmk}
Using the same set of snapshots, we can perform the online simulation with $\Delta t=0.05$ and $U_2$. The results for the optimal control and cost functional can be found in Figure \ref{fig5}.

\begin{figure}[htbp]	
\centering
	\includegraphics[scale=0.4]{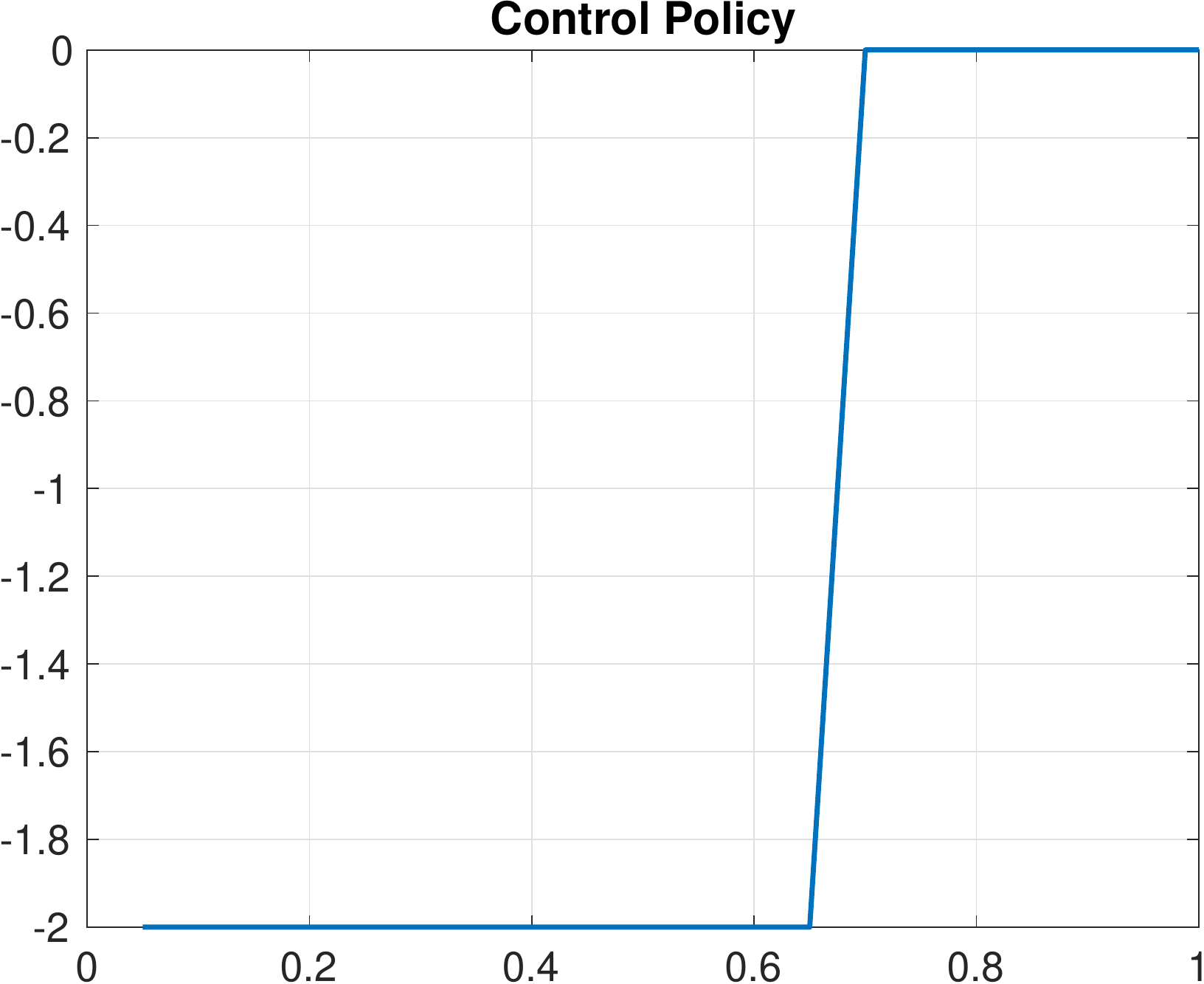}	
	\includegraphics[scale=0.4]{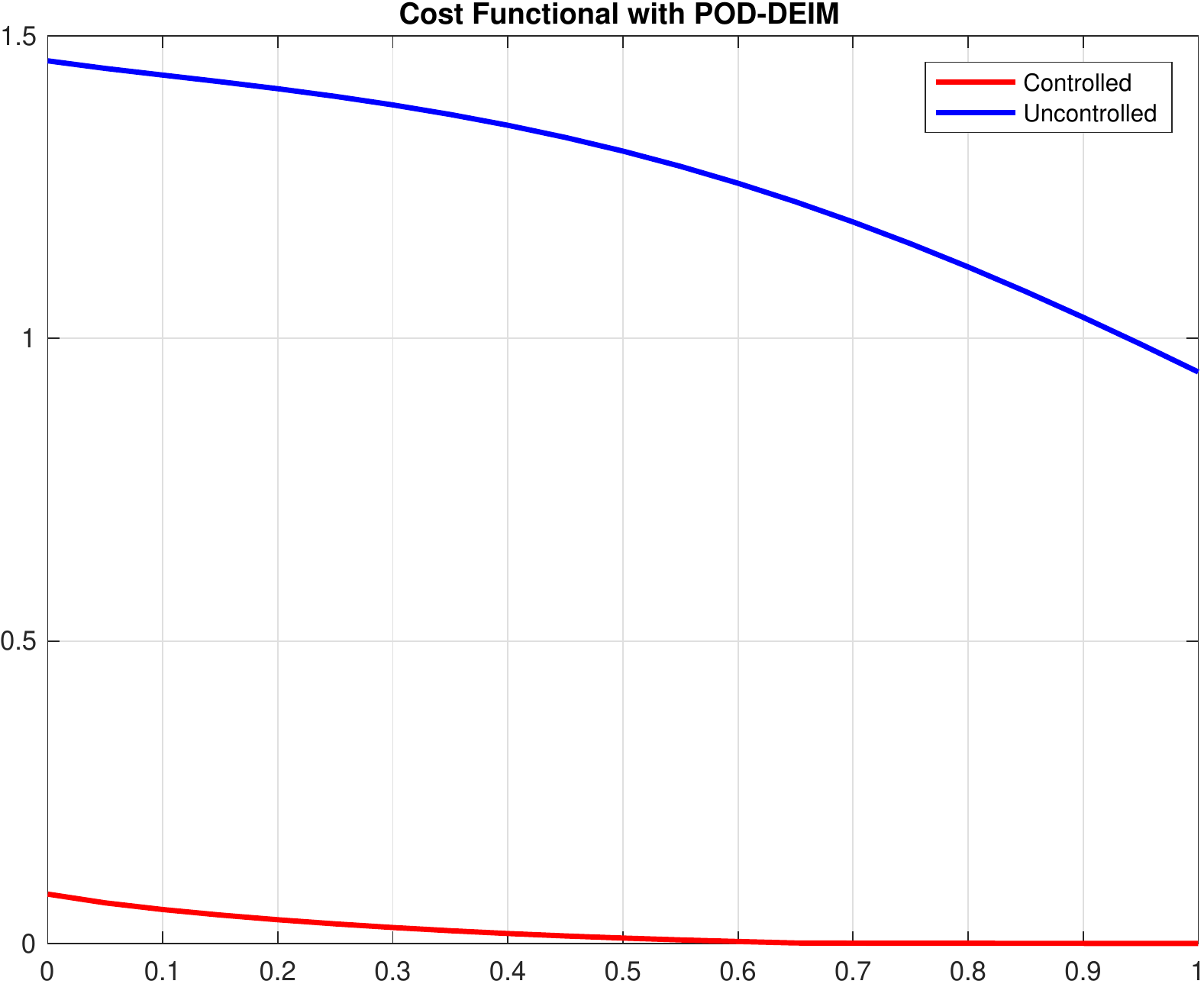}	
\caption{Test 1: Optimal policy (left) and cost functional (right) with $\Delta t=0.05$ and $U_2$.}	
\label{fig5}
\end{figure}

\end{rmk}

\subsection{Test 2: Viscous Burgers' equation}
In the second example we consider the well-known viscous Burgers' equation with homogeneous Dirichlet boundary conditions: 
\begin{equation}
\begin{cases}
\partial_s y(x,s)= \sigma \Delta y(x,s) +y(x,s) \cdot \nabla y(x,s)  +  y(x,s)u(s) & (x,s) \in \Omega \times [0,T],
\\
 y(x,s) =0 & (x,s) \in \partial \Omega \times [0,T], \\
y(x,0)=y_0(x) & x \in \Omega,
\end{cases}
\label{pde2}
\end{equation}
where the control $u(t)$ is taken in the admissible set $\mathcal{U}=\{u:[0,T]\rightarrow [-2,0] \}$ and $\Omega=[0,1]^2$.
In \eqref{pde2} we consider: $T=1, \sigma=0.01$ and $y_0(x_1,x_2)=sin(\pi x_1)sin(\pi x_2).$ We discretize the space domain in $41$ points in each direction, obtaining a problem of dimension $d=1681$ points.
In Figure \ref{fig2pde} we show the solution of the uncontrolled equation \eqref{pde2} for different time instances. Our aim is to steer the solution to the steady state $\tilde{y}(x)=0$, using the cost functional \eqref{cost_test}, as in Test 1, using a bilinear control, e.g. controlling the system through a reaction term.

\begin{figure}[htbp]	
\centering
\centering
	\includegraphics[scale=0.22]{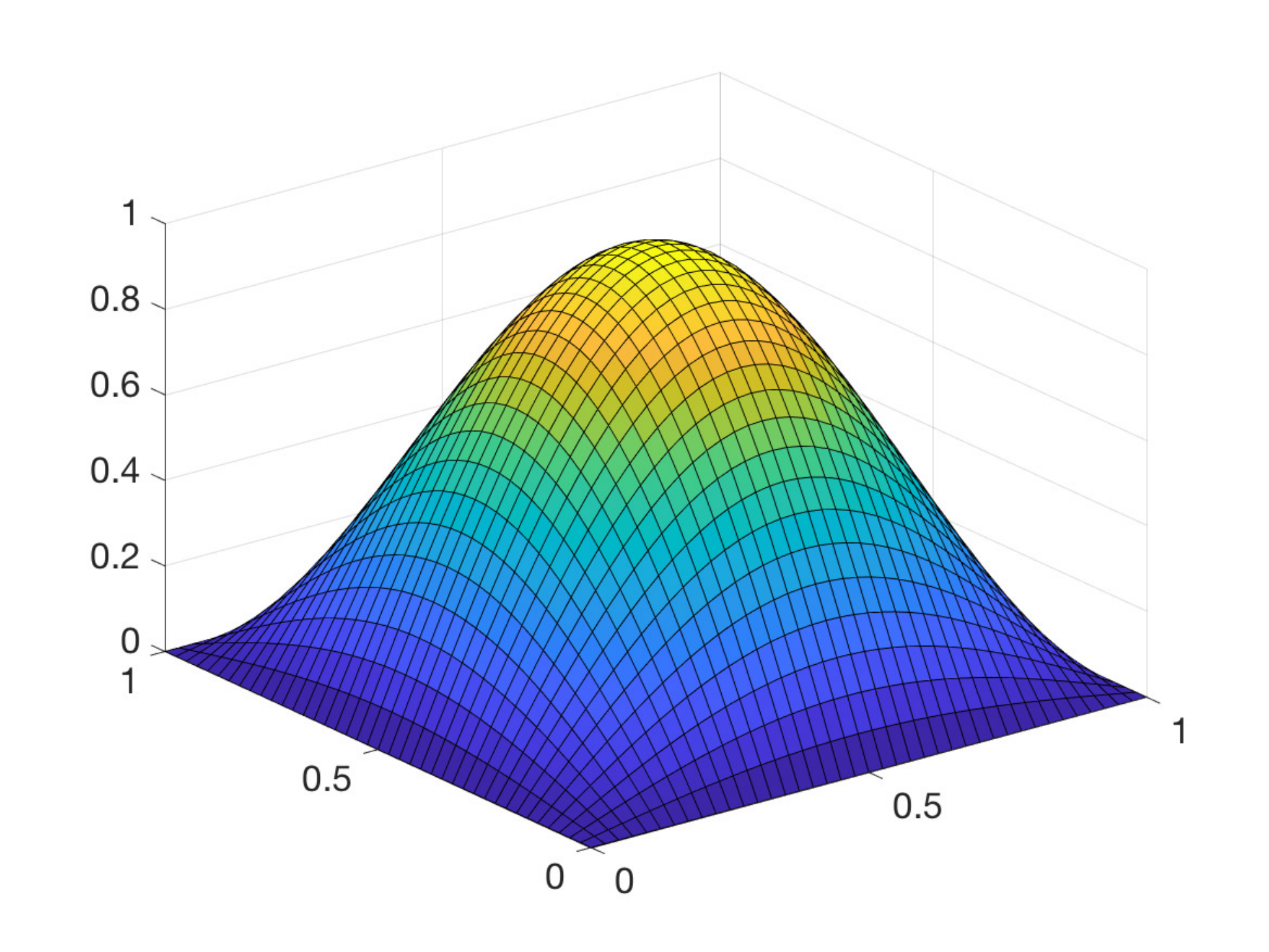}	
	\includegraphics[scale=0.22]{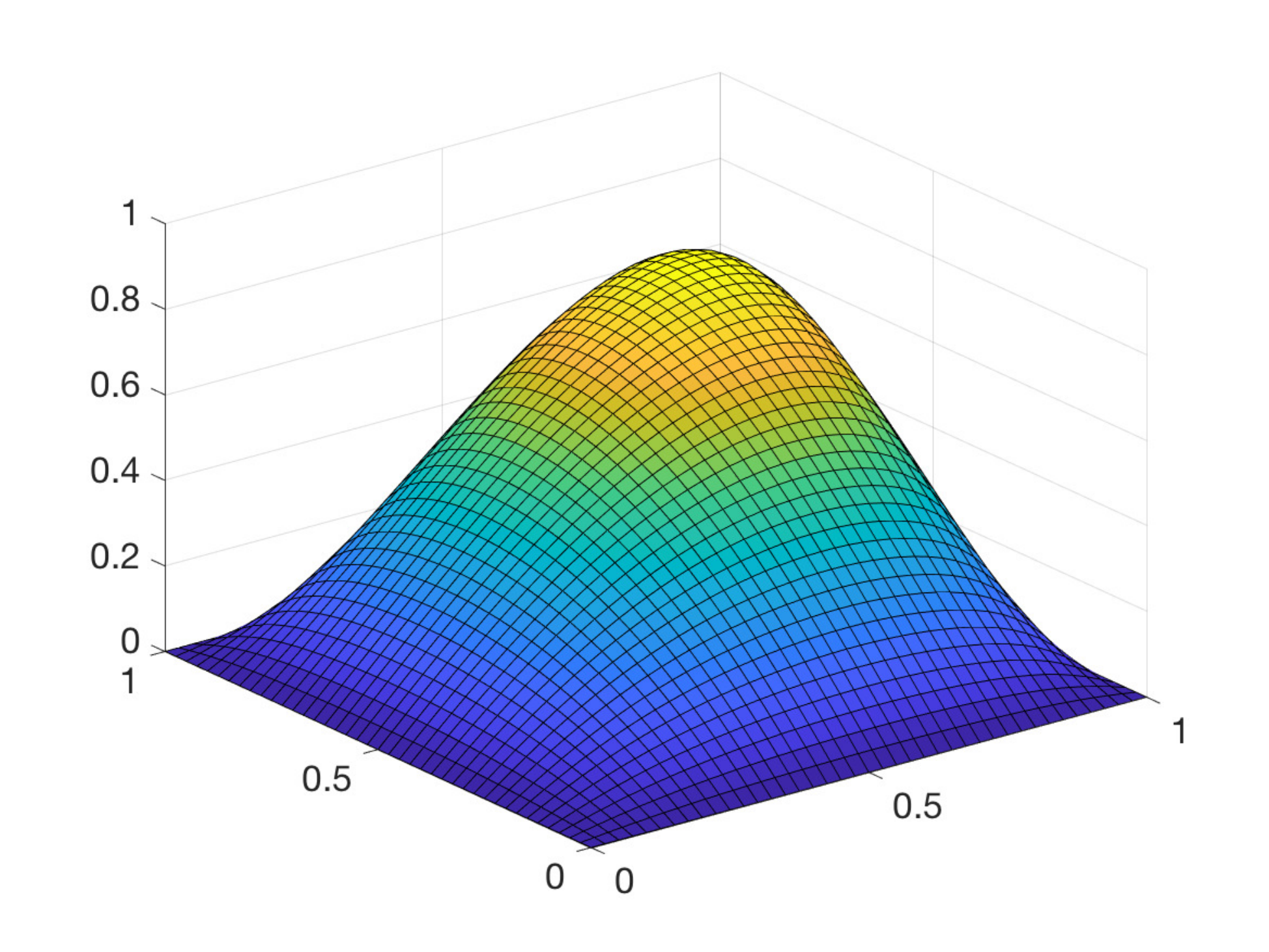}
	\includegraphics[scale=0.22]{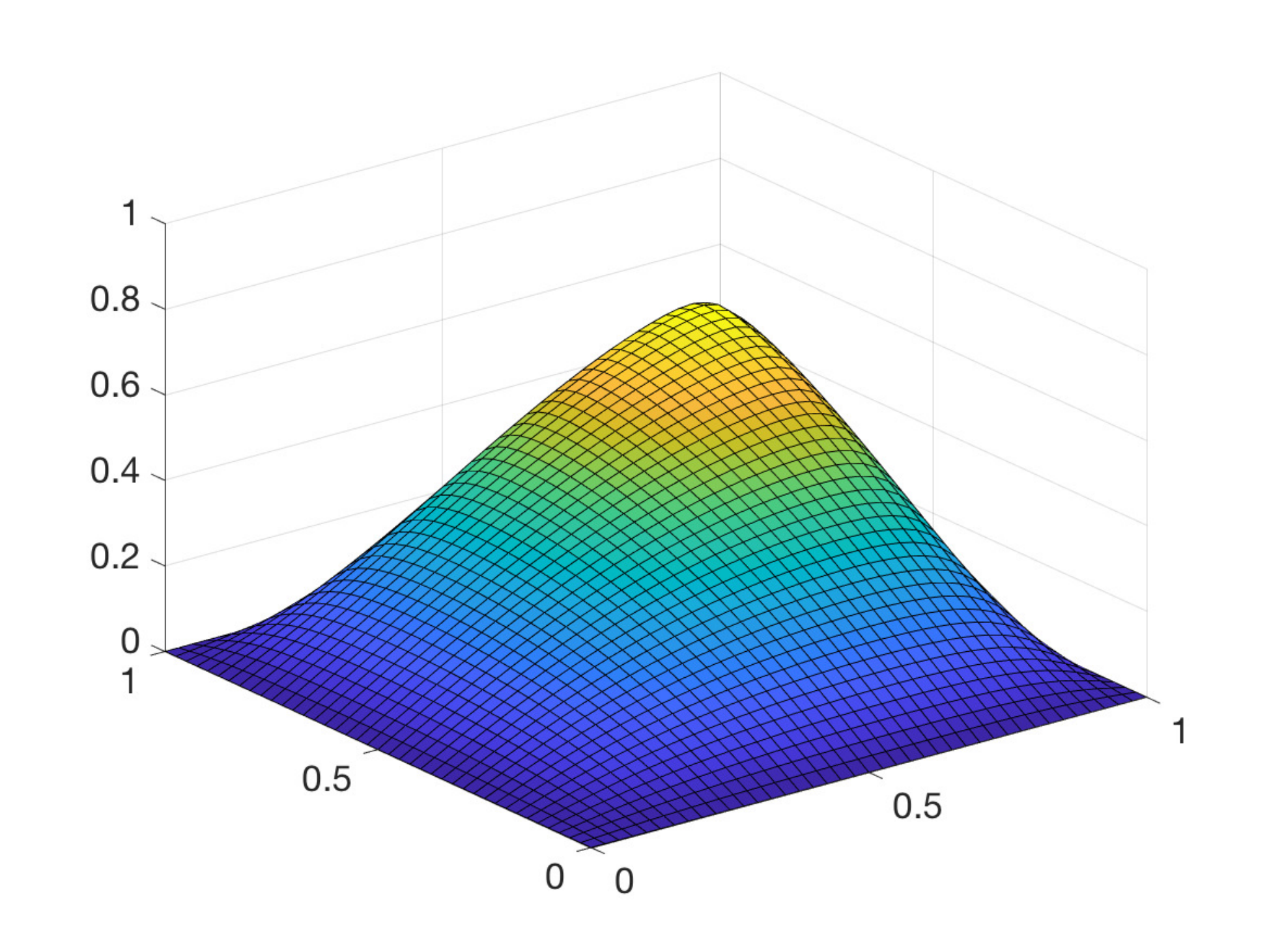}
\caption{Test 2: Uncontrolled solution for equation \eqref{pde2} for time instances $t=\{0, 0.5, 1\}$ (from left to right).}
\label{fig2pde}	
\end{figure}

\paragraph{Case 1: Full TSA} Let us first consider the results of the full TSA. In Figure \ref{fig2pde_c} we show the results of the controlled problem. As we can see, the solution gets close to $\tilde{y}(x)$ as expected. We also note that for this example the viscosity term $\sigma$ is rather low, making the problem hard to be controlled.

\begin{figure}[htbp]	
\centering
\centering
	\includegraphics[scale=0.22]{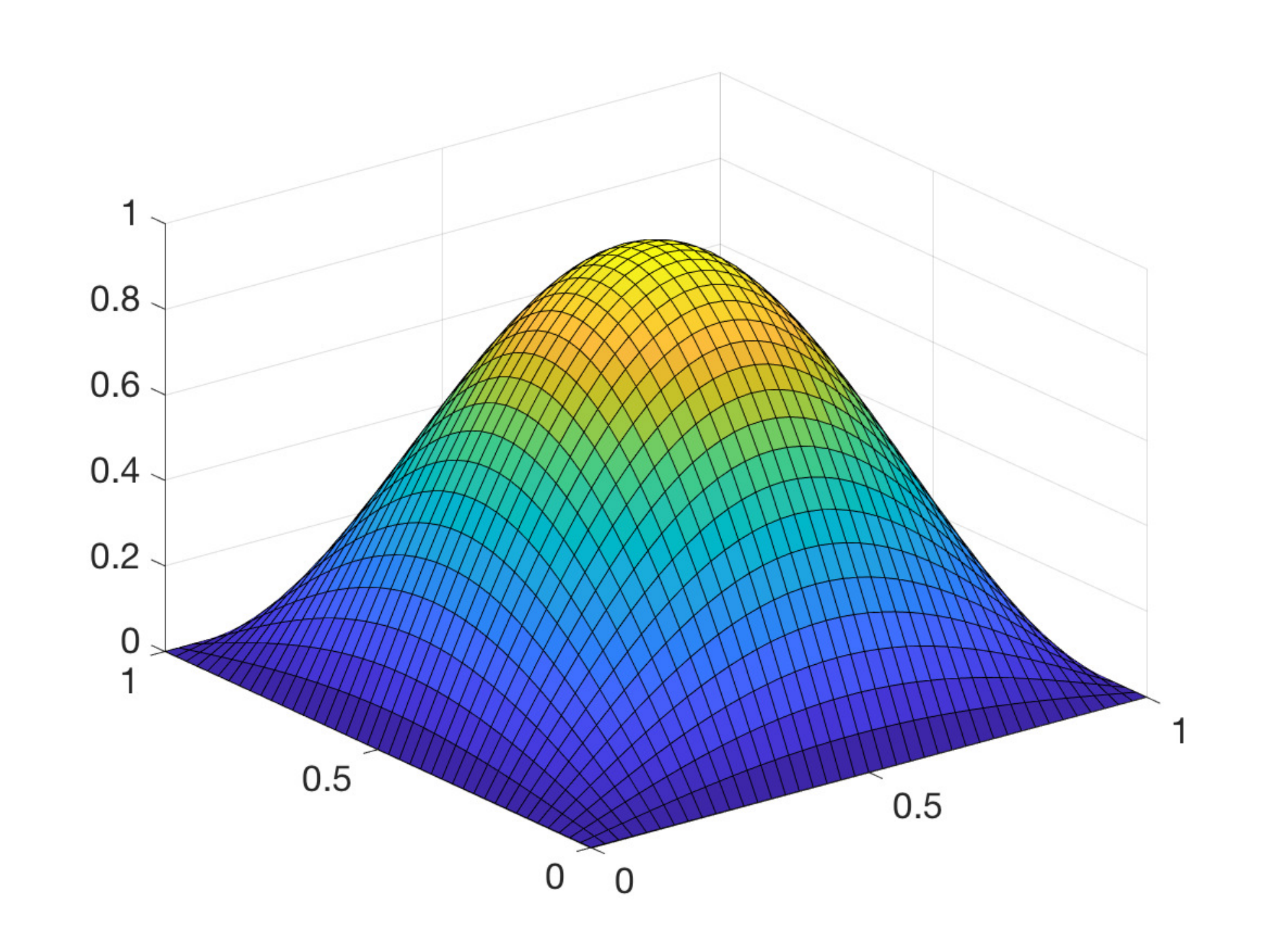}	
	\includegraphics[scale=0.22]{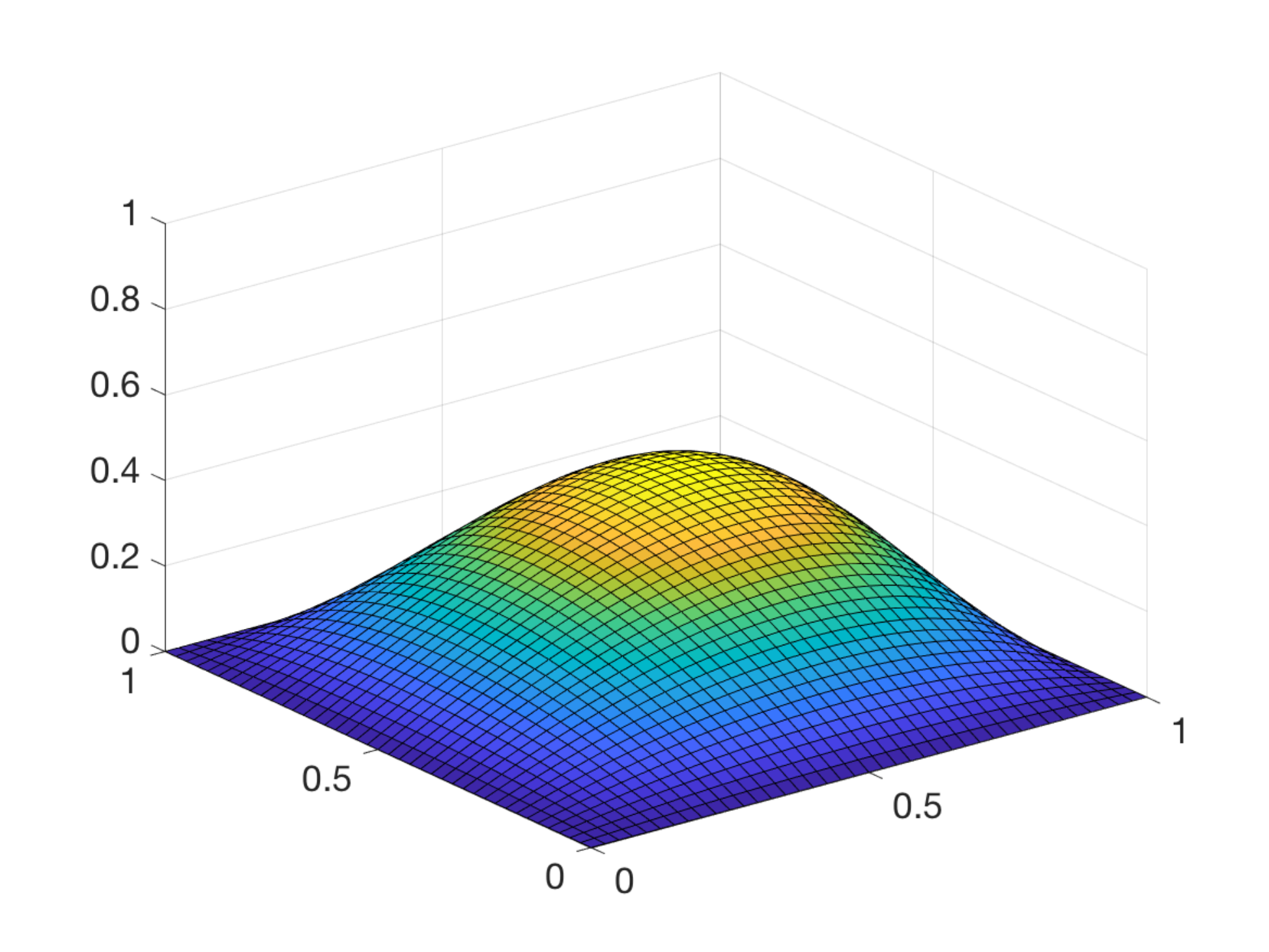}
		\includegraphics[scale=0.22]{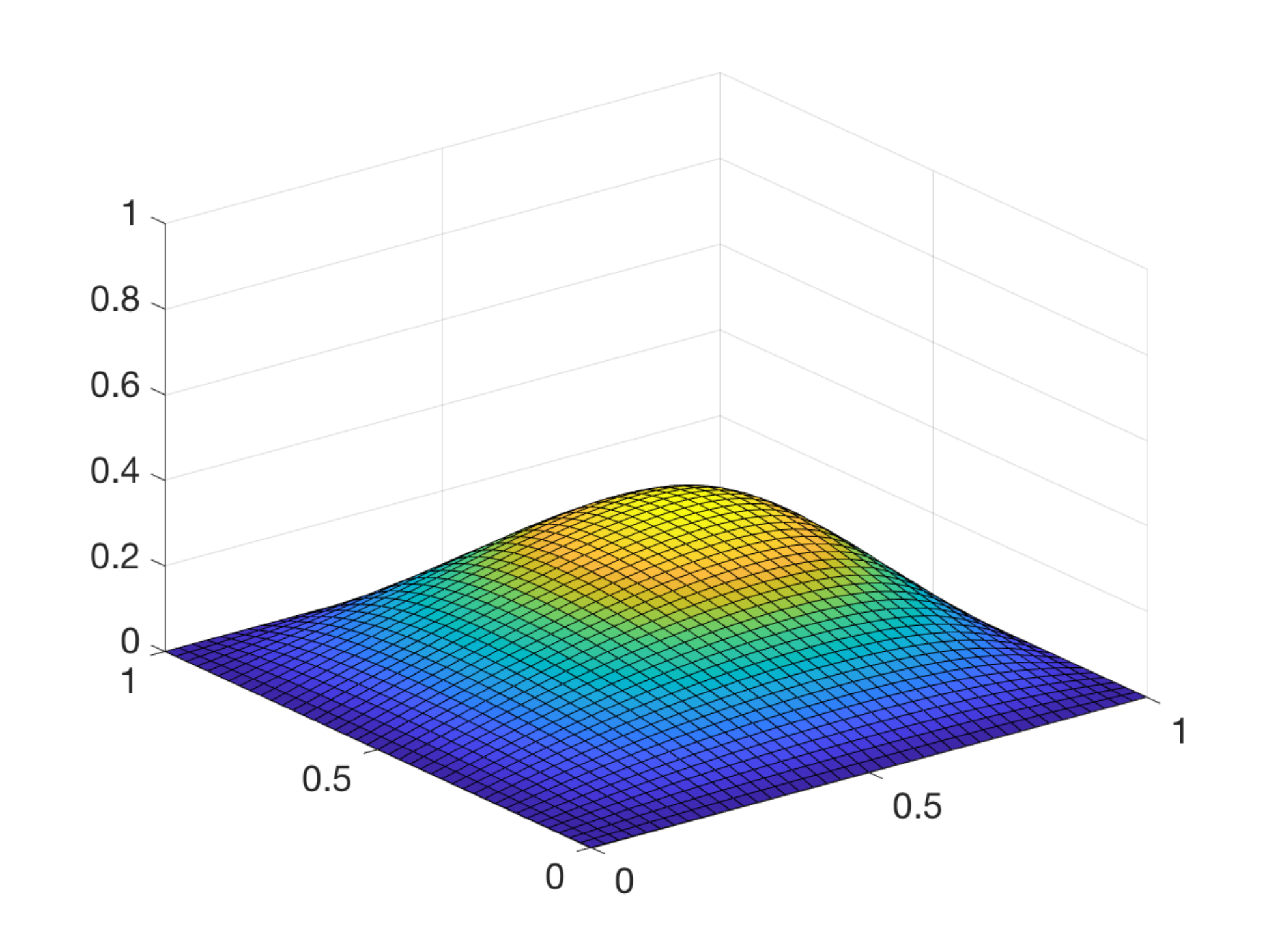}
\caption{Test 2: Controlled solution with $3$ controls for equation \eqref{pde2} with full tree for time $t=\{0, 0.5, 1\}$ (from left to right).}
\label{fig2pde_c}	
\end{figure}

In the left panel of Figure \ref{fig3_2}, we show the optimal control computed to obtain the controlled solution. When the control set is only given by $2$ controls, the algorithm uses the control $u^*(t)=-2$ for $0\leq t \leq 0.7$ and $u(t) = 0$ for $0.7 <t \leq 1$, whereas with $3$ controls we use the control $-2$ for  $0\leq t \leq 0.5$ and $-1$ for $0.5 <t \leq 1$.

We can see that passing from $2$ to $3$ controls, we obtain a slightly better result in terms of cost functional (see 
the right panel of Figure \ref{fig3_2}).
\begin{figure}[htbp]	
\centering
	\includegraphics[scale=0.4]{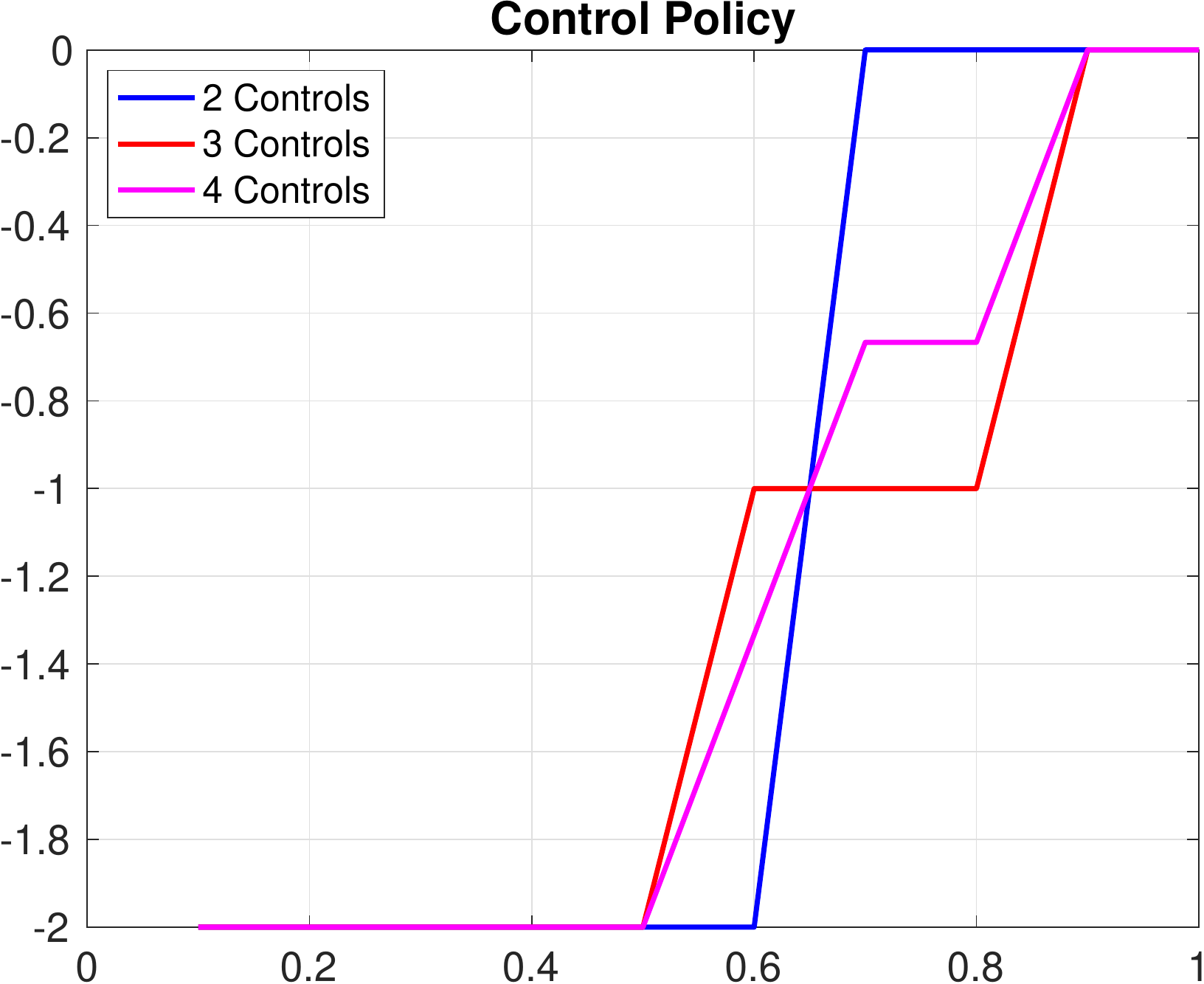}	
	\includegraphics[scale=0.4]{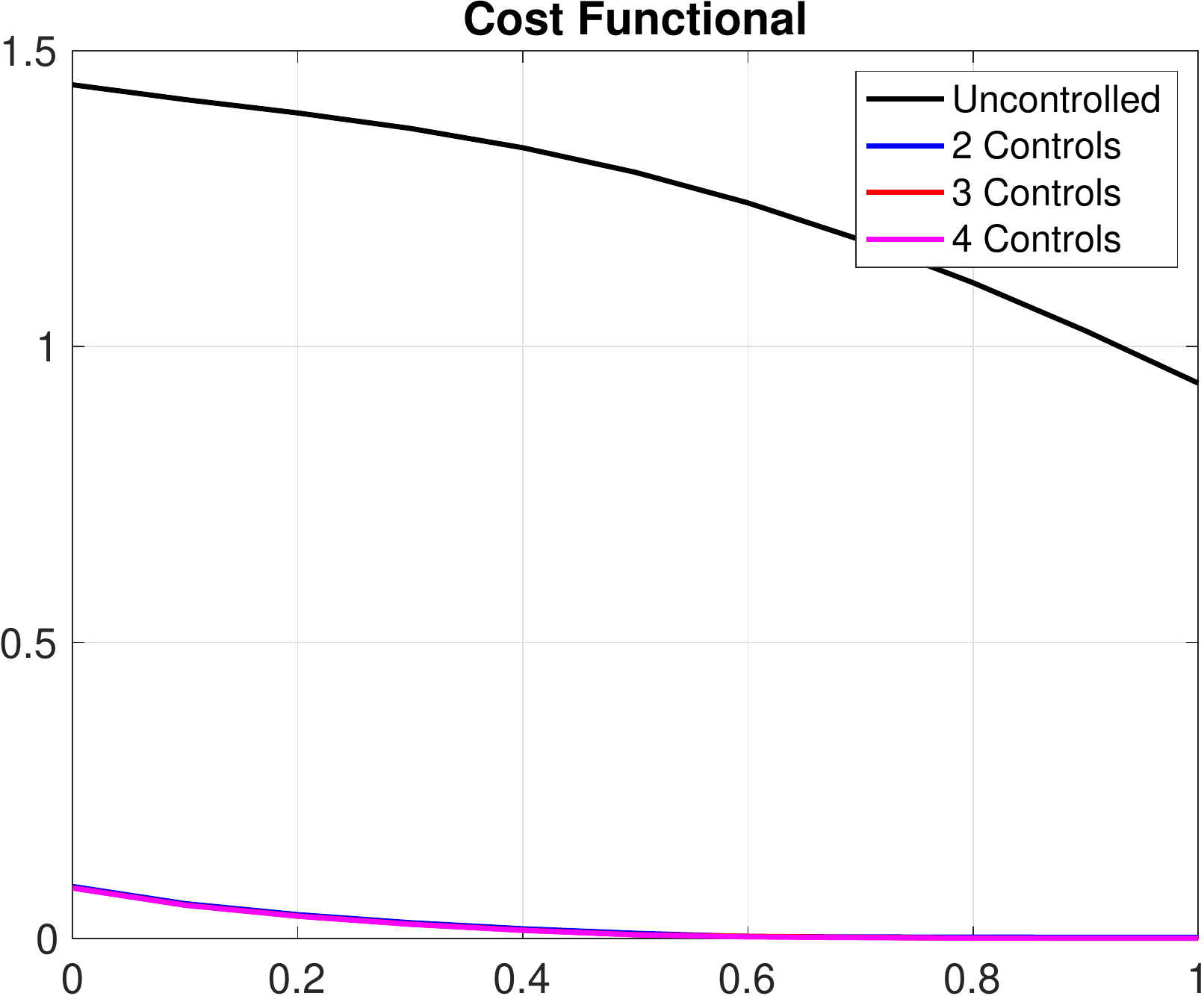}	
\caption{Test 2: Optimal policy (left) and cost functional for $U_2$ and $U_3$ (right).}	
\label{fig3_2}
\end{figure}


Finally, the cardinality of the full tree is reported in Table \ref{tab2:card}. We can observe that the tree is considerably pruned compared to the previous example. This happens when we deal with a bilinear control for both Test 1 and Test 2.

\begin{table}[htbp]
\centering
\begin{tabular}{c c c c c }
& $U_2$ & $U_3$  & $U_4$ & $U_5$ \\
\hline
TSA with $\ep=0$ & $2047$  & $88573$ &  &\\
TSA with $\ep=0.01$ & $1681$ &  $17680$ & & \\
TSA-POD with $\ep=0.01$ & $1717$ & $17627$ & $48372$ &   $83201$\\
\hline
\end{tabular}
\caption{Test 2: Cardinality of the tree for the full TSA and for the pruned TSA and pruned TSA-POD with $\ep =  0.01$, varying the control sets.}
\label{tab2:card}
\end{table}

\paragraph{Case 2: TSA with POD} To accelerate and use a finer control set, we use model order reduction. The snapshots are computed with $\Delta t=0.1$, $U_2$ and a pruning criteria with $\ep=\Delta t^2$. For this problem, we only project the dynamics with POD since the nonlinear term can be written as a tensor and can be projected offline. We took $\ell=8$ POD basis to have $\mathcal{E}(\ell)=0.999$. Thanks to the reduced problem, we are able to solve the problem with more controls, keeping $\Delta t = 0.1$. In the top-left panel of Figure \ref{fig4_2} we show the behaviour of the optimal policy. We note that the cases with $2$ and $3$ controls are equivalent to the full case (compare with Figure \ref{fig3_2}). The computed controls show a chattering behaviour which is then reflected in the plot of $J_{y_0,0}$ in the top-right panel of Figure \ref{fig4_2}, considering the control space $U_n$ for $n=\{2,3,\ldots,11\}$. We can see a rather similar behaviour when increasing the number of controls.

\begin{figure}[htbp]	
\centering
	\includegraphics[scale=0.3]{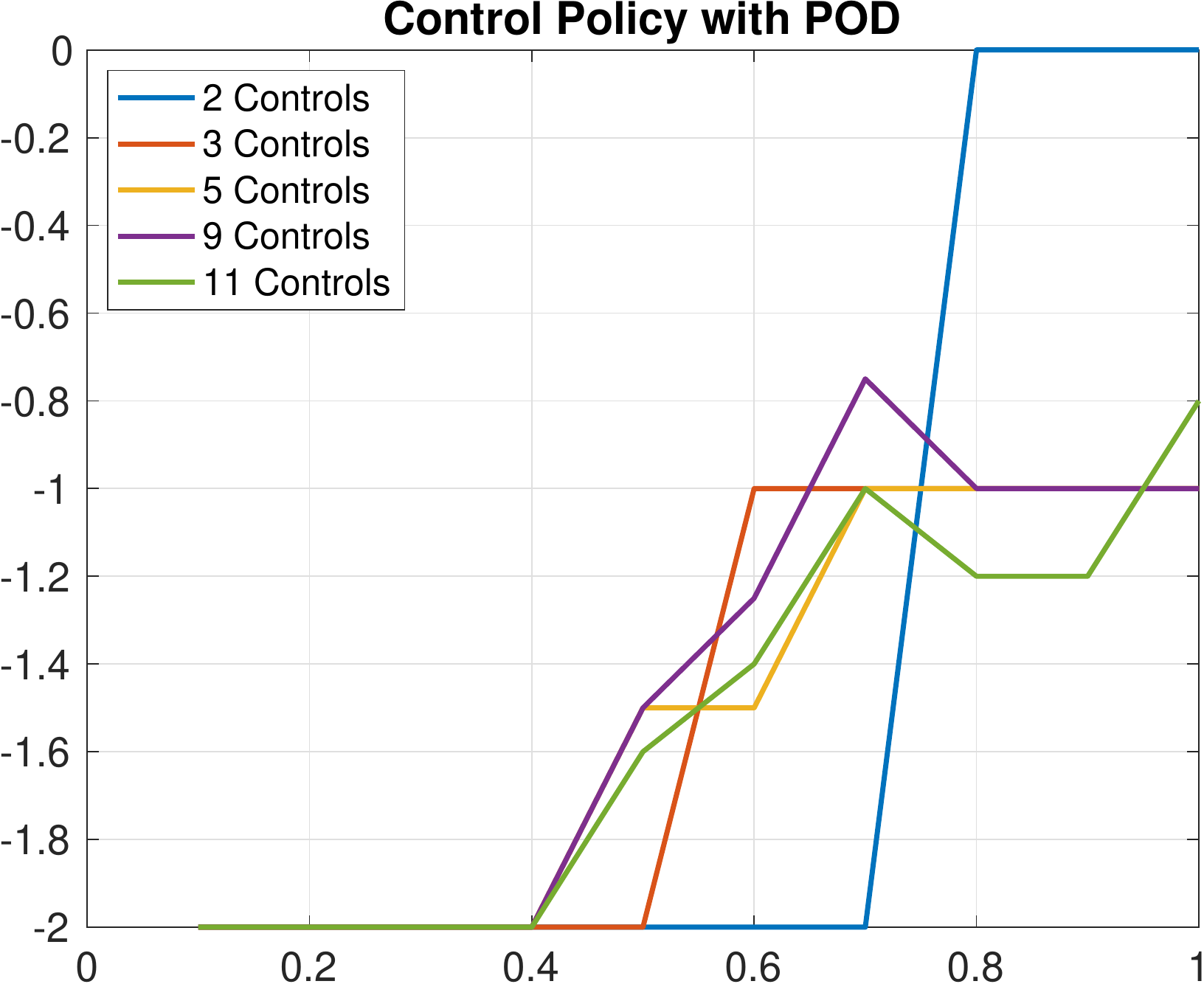}	
	\includegraphics[scale=0.3]{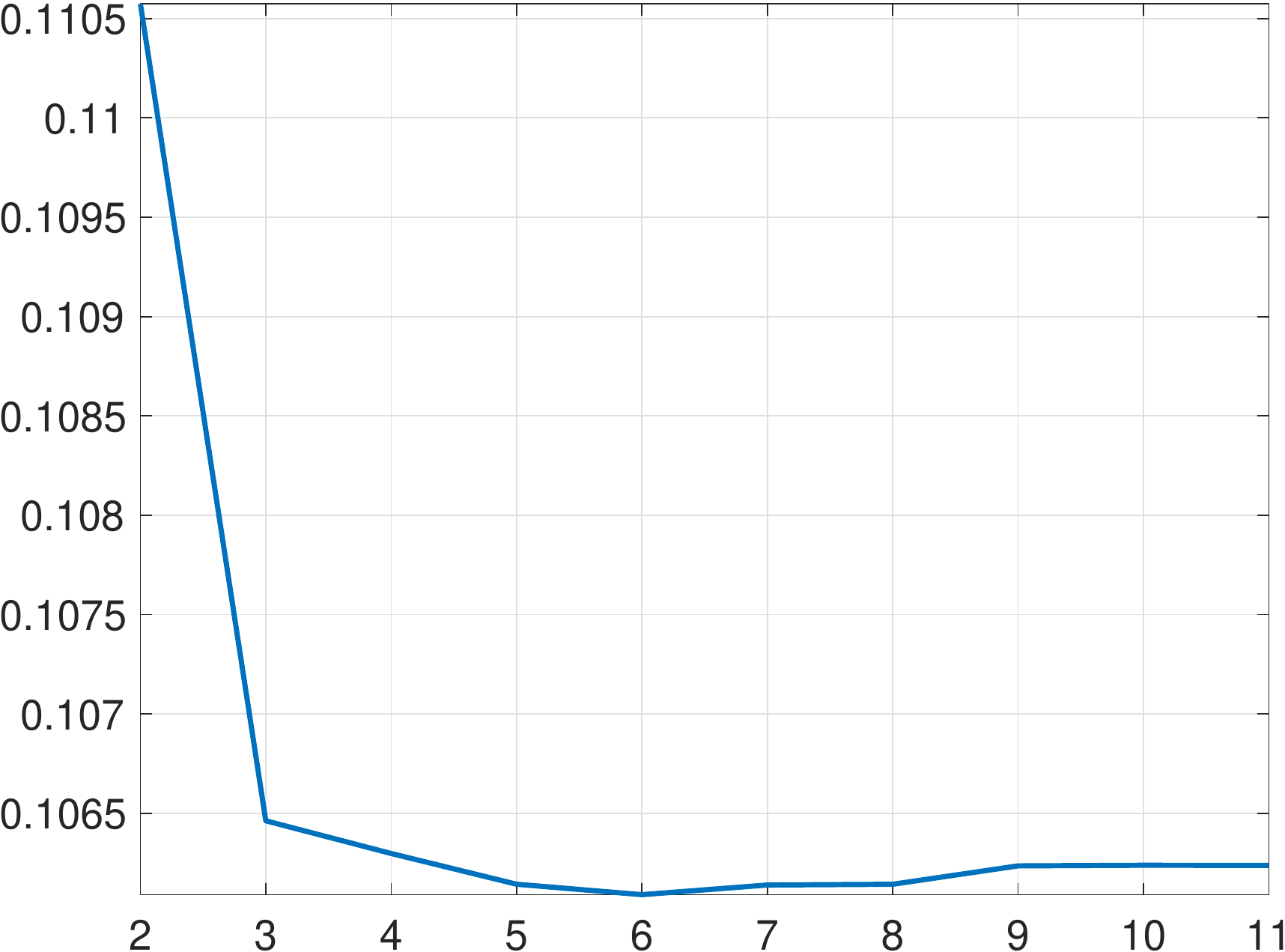}
		\includegraphics[scale=0.3]{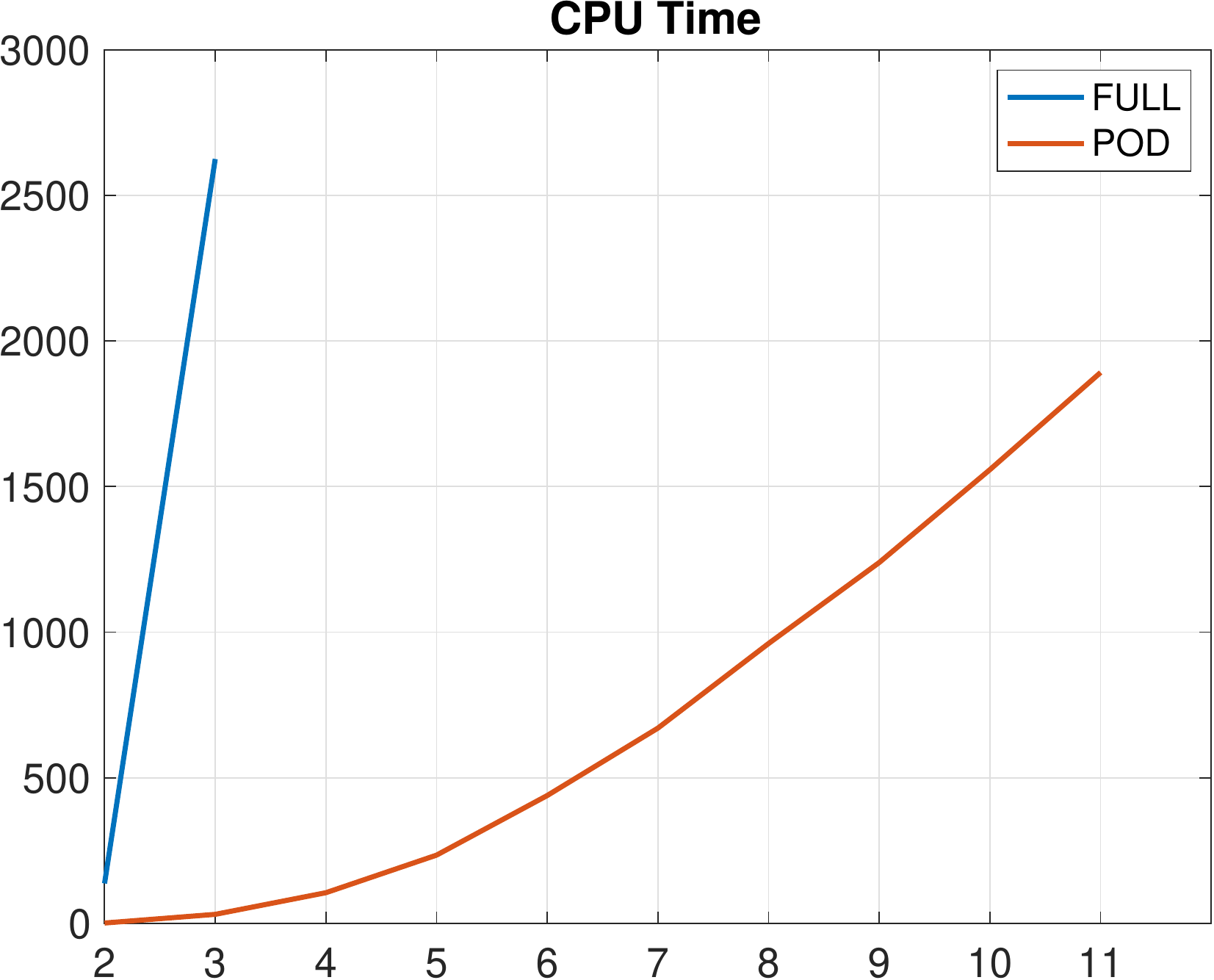}	
\caption{Test 2: Optimal policy (top-left), zoom of the cost (top-right) for $U_n$ with $n=2,3,4,\ldots,11$, and CPU time increasing the number of controls (bottom).}		
\label{fig4_2}
\end{figure}

The CPU time is reported in the bottom panel of Figure \ref{fig4_2} and it is possible to capture visually the big advantage of using model order reduction.

With the same set of snapshots, we can also decrease the temporal step size, e.g. $\Delta t= 0.05$, and compute the online stage with $U_n$ with $n=2,3.$ We see in Figure \ref{fig4_3} that the behaviour of the control policy is similar when dealing with $2$ controls, whereas the switch from $u=-2$ to $u=-1$ happens for $t=0.45$.

\begin{figure}[htbp]	
\centering
	\includegraphics[scale=0.4]{Pic_contr2_deim.pdf}	
	\includegraphics[scale=0.4]{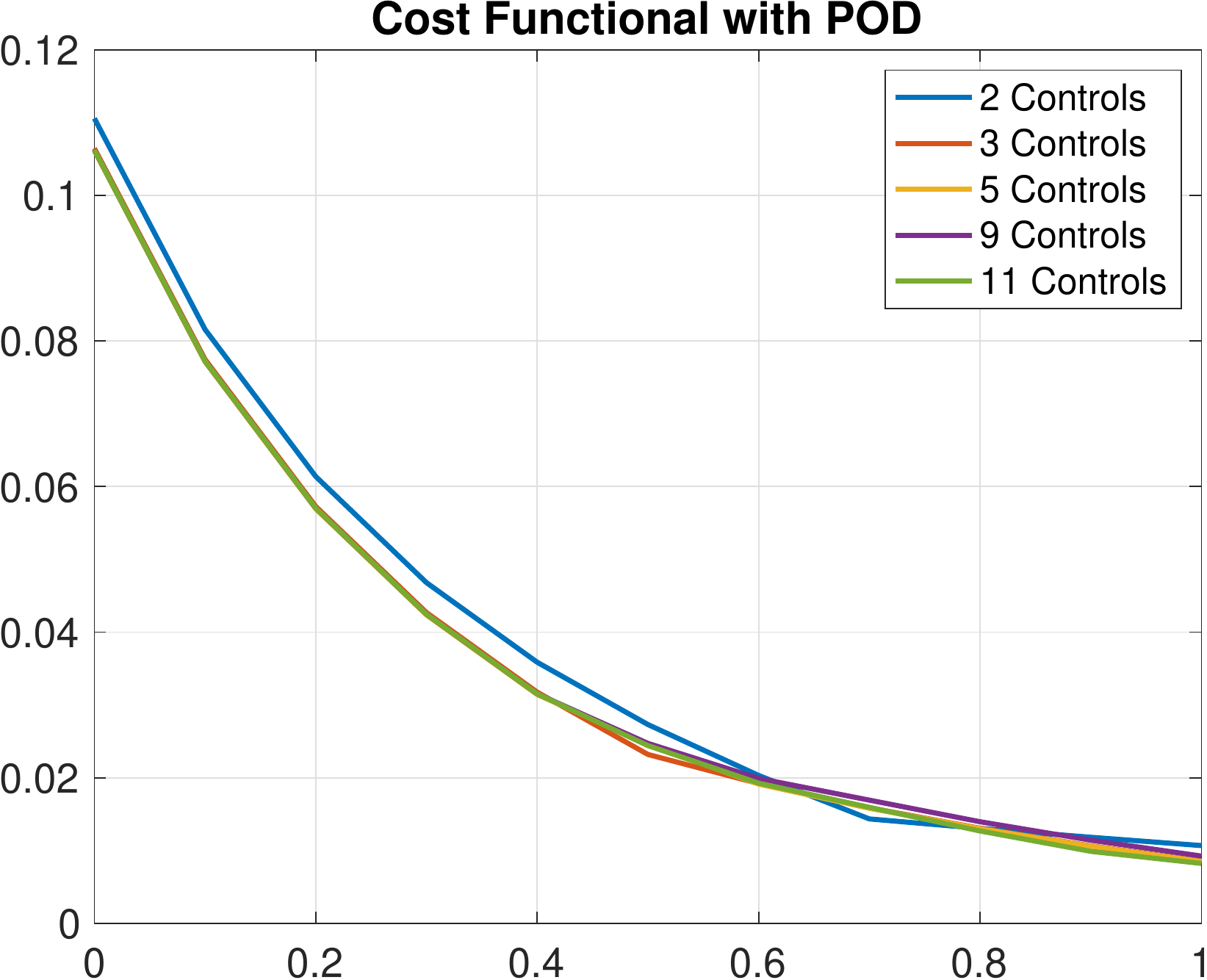}	
\caption{Test 2: Optimal policy (left) and cost functional for $U_n$ with $n=2,3$ and $\Delta t= 0.05$ (right).}	
\label{fig4_3}
\end{figure}

We can also observe that the cost functional is slightly lower when dealing with $\Delta t =0.05$ as summarized in Table \ref{tab3:cost}.

\begin{table}[htbp]
\centering
\begin{tabular}{c c c  }
$\Delta t$& $U_2$ & $U_3$  \\
\hline
0.1 & $0.1106$  & $0.1065$ \\
0.05 & $0.0995$ &  $0.0956$ \\
\hline
\end{tabular}
\caption{Test 2: Cost functional $J^\ell_{y_0,0}$ with $\Delta t \in \{0.1,0.05\}$, $U_2$ and $U_3$.}
\label{tab3:cost}
\end{table}

The cardinality of the pruned TSA-POD approach is reported in the last line of Table \ref{tab2:card}, whereas the first line is still valid for the full TSA-POD method. As expected, even when we apply model reduction, we can observe an impressive pruning if we compare with the unpruned method.

\section{Conclusions and future works}\label{sec:con}

In this work we have presented a new method that couples model order reduction with a recent technique to solve DP approach on a tree structure, proposed in \cite{AFS18, SAF18}. The tree structure needs to solve many PDEs for a given control input and, therefore, model order reduction helps to speed up its construction and also to work with a finer control set. We have also provided an error estimate to guarantee the convergence of the method which depends, as expected, on the projection error of the POD method and on the temporal discretization of the differential equations considered.
We showed through numerical tests the efficiency of the method and we would like to emphasize that the tree structure algorithm combined with model order reduction allows to solve numerical optimal control problems for nonlinear PDEs.

Some limitations of the method will be addressed in future works. Here, we strongly rely on a finite discretization of the control set. We would like to improve the feedback reconstruction by means of more sophisticated methods which do not need a finite number of controls. That will also avoid the use of a comparison method in the computation of the minimum of the hamiltonian. Clearly, the pruning rule help to reduce the dimension of the tree and to keep its cardinality feasible. 
Another interesting future application is the extension of the tree structure to stochastic control problems. 

\end{document}